\documentclass[11pt, reqno]{amsart}

\numberwithin{equation}{section}

\usepackage{amscd,amsmath}
\usepackage{mathrsfs}
\usepackage{amsfonts}
\usepackage{amssymb}
\usepackage{enumerate}
\usepackage{setspace}
\usepackage{color}

\usepackage[margin=1in]{geometry}

\usepackage[colorlinks=true, citecolor=blue]{hyperref}

\newtheorem{theorem}{Theorem}[section]
\newtheorem{lemma}{Lemma}[section]

\newtheorem{corollary}{Corollary}[section]

\newtheorem{remark}{Remark}[section]
\newtheorem{proposition}{Proposition}[section]

\newcommand{\eqn}{\begin{eqnarray}}
\newcommand{\een}{\end{eqnarray}}

\DeclareMathOperator{\dv}{div}
\DeclareMathOperator{\curl}{curl}

\makeatletter
 \newcommand{\norm}{\@ifstar{\@normb}{\@normi}}
 \newcommand{\@normb}[2]{\left\Vert{#1}\right\Vert_{#2}}
 \newcommand{\@normi}[2]{\Vert{#1}\Vert_{#2}}
 \makeatother

\begin{document}

\title[Hall MHD]{On the regularity of magneto-vorticity field and the global existence for the Hall magnetohydrodynamic equations}

\author{Hantaek Bae}
\address{Department of Mathematical Sciences, Ulsan National Institute of Science and Technology (UNIST), Republic of Korea}
\email{hantaek@unist.ac.kr}

\author{Kyungkeun Kang}
\address{Department of Mathematics, Yonsei University, Republic of Korea}
\email{kkang@yonsei.ac.kr} 

\author{Jaeyong Shin}
\address{Department of Mathematics, Yonsei University, Republic of Korea}
\email{sinjaey@yonsei.ac.kr}

\date{\today}

\begin{abstract}
In this paper, we investigate the incompressible viscous and resistive Hall magnetohydrodynamic equations (Hall MHD in short). We first study the regularity of the magneto-vorticity field $B+\omega$. In three dimensions, we derive some bounds of $B+\omega$ under a condition of the velocity field $u$. Moreover, if we consider the Hall MHD with 2D variables, the uniform-in-time bounds of $B+\omega$ come from the three dimensional case. The regularity of $B+\omega$ gives us a crucial clue of blow-up scenario and provides conditions of the existence of global-in-time  solutions. In particular, we prove the global well-posedness of the Hall MHD (also the electron MHD) with 2D variables  when the third  component of the initial current density $J_0=\nabla\times B_0$ is sufficiently small. We also derive temporal decay rate of $B+\omega$.
\end{abstract}

\maketitle

\vspace{-1ex}

%%%%%%%%%%%
\section{Introduction}
%%%%%%%%%%%
In this paper, we investigate the incompressible viscous and resistive Hall magnetohydrodynamic equations (Hall MHD in short):
\begin{subequations}\label{H MHD}
\begin{align}
&\partial_{t} u-\nu\Delta u+u\cdot\nabla u-(\nabla \times  B)\times B+\nabla p=0, \label{H MHD a}\\
&\partial_{t} B-\eta\Delta B+u\cdot\nabla B-B\cdot\nabla u+h\nabla \times ((\nabla \times  B)\times B)=0,  \label{H MHD b}\\ 
&\dv u=0,\quad \dv B=0,
\end{align}
\end{subequations}
where $u$ is the  velocity field, $p$ is the pressure, and $B$ is the magnetic field. The term $(\nabla \times  B)\times B$ is called the Lorentz force and  $\nabla \times ((\nabla \times  B)\times B)$ is called the Hall term.   The positive constants  $(\nu, \eta, h)$ and called the viscosity, the resistivity, and  the Hall  constants, respectively.  The Hall term is introduced by \cite{Lighthill} to add the influence of the electric current in the Lorentz force.

The Hall MHD is important in describing many physical phenomena \cite{Balbus, Forbes, Homann, Lighthill, Mininni, Shalybkov, Shay, Turner, Wardle}. The Hall MHD has been actively studied mathematically: see \cite{Acheritogaray, Bae Kang, Chae Degond Liu, Chae Lee, Chae Schonbek, Chae Wan Wu, Chae Weng, Chae Wolf, Chae Wolf 1, Chae Wolf 2, Dai 1, Dai 2, Dai 3, Danchin Tan 2, Danchin Tan 3, Jeong Oh, Kwak, Yamazaki} and references therein.

In this paper, we consider the Hall MHD with three and two dimensional space variables. The goal of this paper is threefold: (i) regularity of the magneto-vorticity field $B+h\nabla\times u$ and its applications; (ii) global existence of the Hall MHD with two dimensional variables when the third component of the initial current density $J_0=\nabla\times B_0$ is sufficiently small; (iii) temporal decay rates of the magneto-vorticity $B+\omega$ in $L^{2}(\mathbb{R}^{3})$. For (i) and (ii), we set initial data in $H^{d}(\mathbb{R}^d)$, $d=2,3$, which are not optimal to construct solutions to the Hall MHD, but we use these spaces for simplicity.

%%%%%%%%%%%%%%%%%%%%%%%%%
\subsection{Magneto-vorticity field $B+h\omega$} \label{sec:1.1}
%%%%%%%%%%%%%%%%%%%%%%%%%%
The starting point of this paper is to study the regularity of the magneto-vorticity field $B+h\omega$, where $\omega=\nabla \times  u$. To do so, we first derive the equation of $B+h\omega$. By taking the curl operator to $h$(\ref{H MHD a}) and adding this to (\ref{H MHD b}), we can remove the Hall term. Consequently, we arrive at the following equations
\eqn \label{eq:B+omega}
\begin{split}
&\partial_{t}(B+h\omega)-\nu \Delta (B+h\omega)+u\cdot \nabla (B+h\omega) -(B+h\omega)\cdot \nabla u=(\eta-\nu)\Delta B,\\
& \dv (B+h\omega)=0.
\end{split}
\een
Since the argument used in this paper does not depend on the size of $h$, we set $h=1$. When $\nu=\eta>0$, the right-hand side of (\ref{eq:B+omega}) disappears and thus $B+\omega$ looks like the vorticity form of the incompressible Navier-Stokes equations. So we can expect more regularity of $B+\omega$ than each of $B$ and $\omega$ when $\nu=\eta>0$.

The quantity $B+\omega$, which is called the magneto-vorticity field \cite{Poly}, is already used in several mathematical work to study (\ref{H MHD}): for example, regularity and partial regularity when $\nu=\eta>0$ \cite{Chae Wolf, Chae Wolf 1, Chae Wolf 2}; small data global existence with 2D variables when $\nu=\eta>0$ \cite{Danchin Tan 3};  singularity formation with axisymmetry \cite{Chae Weng} and ill-posedness \cite{Jeong Oh} when $\nu=\eta=0$. The condition $\nu=\eta>0$  is also used in many numerical simulations \cite{Meyrand, Mininni}.

%%%%%%%%%%%%%%%%%%%%
\subsubsection{\bf Regularity of $B+\omega$ in 3D}
%%%%%%%%%%%%%%%%%%%
We first recall the energy inequality
\begin{equation}\label{eq:energy-ineq}
\norm{u(t)}{L^2}^{2}+\norm{B(t)}{L^2}^{2}+2\nu\int^{t}_{0}\norm{\nabla u(\tau)}{L^2}^{2}\,d\tau+2\eta\int^{t}_{0}\norm{\nabla B(\tau)}{L^2}^{2}\,d\tau\leq \norm{u_{0}}{L^2}^{2}+\norm{B_{0}}{L^2}^{2}=\mathcal{E}_0. 
\end{equation}
The local well-posedness of \eqref{H MHD} with initial data in $H^3(\mathbb{R}^{3})$  is proved in \cite{Chae Degond Liu}. Moreover, it is shown in \cite{Chae Lee} that the maximal existence time $T^*<\infty$ and so 
\[
\limsup_{t\nearrow T^*}\left(\norm{u(t)}{H^3}^2+\norm{B(t)}{H^3}^2\right)=\infty
\]
if and only if 
\begin{equation}\label{eq:Chae-Lee}
\int^{T^*}_{0}\left(\norm{u(t)}{\text{BMO}}^{2}+\norm{\nabla B(t)}{\text{BMO}}^{2}\right)\,dt=\infty,
\end{equation}
where BMO is the space of functions of bounded mean oscillation \cite{John Nirenberg}. The first result of this paper is about a condition of $u$ to bound $B+\omega$ up to the existence time of $(u,B)$. In fact, the condition (\ref{eq:regularity-cond}) below is same to the blow-up condition of the incompressible Navier-Stokes equations \cite{Kozono Taniuchi}.

\begin{theorem}\label{Theorem 1.1}\upshape
Let $\nu,\eta>0$ and $(u_{0},B_{0})\in H^{3}(\mathbb{R}^3)$ with $\dv u_0=\dv B_0=0$. Let $(u,B)\in C([0,T);H^{3}(\mathbb{R}^3))$ be the solution  of \eqref{H MHD}. If $u$  satisfies
\begin{equation} \label{eq:regularity-cond}
U_T= \int^{T}_{0}\norm{u(t)}{\text{BMO}}^{2}\,dt<\infty,
\end{equation}
then we have the following bounds:
\begin{equation}\label{eq:bound-1}
\begin{split}
&\sup_{t\in[0,T)}\norm{(B+\omega)(t)}{L^2}^{2}+\nu\int^{T}_{0}\norm{\nabla(B+\omega)(t)}{L^{2}}^{2}\,dt \\
&\leq \left(\norm{B_{0}+\omega_{0}}{L^2}^{2}+C\mathcal{E}_0\frac{(\nu-\eta)^{2}}{\nu\eta} \right)\exp{\left(\frac{CU_T}{\nu}\right)}= \mathcal{E}_1(U_T),
\end{split}
\end{equation}
\begin{equation} \label{eq:bound-2}
\sup_{t\in[0,T)}\left\|\omega(t)\right\|^{2}_{L^{2}}+\min{(\nu,\eta)}\int^{T}_{0}\left\|\nabla\omega(t)\right\|^{2}_{L^{2}}\,dt \leq \mathcal{E}_1(U_T)+\mathcal{E}_0=\mathcal{E}_2(U_T).
\end{equation}
Moreover, if $\nu=\eta$, then $B+\omega$ also satisfies
\begin{equation}\label{eq:bound-3}
\sup_{t\in[0,T)}\left\|(B+\omega)(t)\right\|^{2}_{H^{2}}+\nu\int^{T}_{0}\left\|\nabla(B+\omega)(t)\right\|^{2}_{H^{2}}\,dt\leq \norm{B_{0}+\omega_{0}}{H^2}^{2}\exp\left(\frac{C\mathcal{E}_2(U_T)}{\nu^2}\right).
\end{equation}
\end{theorem}

\begin{remark} \label{Remark 1.1} \upshape
There are several remarks in order.
\begin{enumerate}
\item \eqref{eq:regularity-cond} can be replaced by Serrin-type conditions with a minor modification:
\[
U_T=\int^{T}_{0}\norm{u(t)}{L^p}^q\,dt<\infty\quad\text{with}\quad \frac{3}{p}+\frac{2}{q}\leq 1. 
\]
\item \eqref{eq:regularity-cond} is a necessary and sufficient condition of the regularity criterion: when \eqref{eq:bound-2} holds, 
\[
\int^{T}_{0}\norm{u(t)}{\text{BMO}}^{2}\,dt\leq C\int^{T}_{0}\norm{u(t)}{\dot{H}^{\frac{3}{2}}}^{2}\,dt\leq C\int^{T}_{0}\norm{\nabla u(t)}{L^2}\norm{\nabla\omega(t)}{L^2}\,dt<\infty
\]
by (\ref{BMO embedding}), (\ref{curl-Lp}) and (\ref{Interpolation}). However, \eqref{eq:regularity-cond} does not guarantee that the $(u,B)$ can be extended beyond  $t=T$. Nevertheless, Theorem \ref{Theorem 1.1} provides more information about the solution at blow-up time $t=T^*$. In fact, Theorem \ref{Theorem 1.1} with \eqref{eq:Chae-Lee} implies that if $U_{T^{\ast}}=\infty$, then
\begin{align*}
& \sup_{t\in[0,T^*)}\norm{(B+\omega)(t)}{L^2}^{2}+\nu\int^{T^*}_{0}\norm{\nabla(B+\omega)(t)}{L^{2}}^{2}\,dt=\infty,\\
& \sup_{t\in[0,T^*)}\left\|\omega(t)\right\|^{2}_{L^{2}}+\min{(\nu,\eta)}\int^{T^*}_{0}\left\|\nabla\omega(t)\right\|^{2}_{L^{2}}\,dt=\infty.
\end{align*}
On the other hand, if 
\[
U_{T^*}=\int^{T^*}_{0}\norm{u(t)}{\text{BMO}}^2\,dt<\infty \quad\text{and}\quad \int^{T^*}_{0}\norm{\nabla B(t)}{\text{BMO}}^2\,dt=\infty,
\]
(\ref{eq:bound-1}) and (\ref{eq:bound-2}) still hold until $t=T^*$ even if $(u,B)$ blows up at $t=T^*$. If $\nu=\eta>0$, we can  show the blow-up criteria only in terms of $u$ by using \eqref{eq:bound-3}: $T^*<\infty$ if and only if
\[
\int^{T^*}_{0}\norm{\nabla \omega(t)}{\text{BMO}}^2\,dt=\infty,
\]
which is an alternative proof of the result in \cite{Ye}. 
\item When $\nu=\eta>0$, Theorem \ref{Theorem 1.1} also implies 
\[
\sup_{t\in[0,T)}\left\|\partial_{t}(B+\omega)(t)\right\|_{L^{2}} \leq C(\nu+\sqrt{\mathcal{E}_2(U_T)}) \norm{B_{0}+\omega_{0}}{H^2}\exp\left(\frac{C\mathcal{E}_2(U_T)}{\nu^2}\right).
\]
\item We now consider \eqref{H MHD} with $\nu=\eta=0$. Then, we can find a similar regularity condition even though \eqref{H MHD} is not locally well-posed in general  \cite{Jeong Oh}. Formally, if
\[
A(T)=\int^{T}_0\norm{\nabla u(t)}{L^{\infty}}\,dt<\infty,
\]
then the classical solution $(u,B)$ of \eqref{H MHD} (if it exists) satisfies
\[
\begin{aligned}
\sup_{t\in[0,T)} \norm{(B+\omega)(t)}{L^p} &\leq \norm{B_0+\omega_0}{L^p} e^{CA(T)},  \  \ 2\leq p\leq \infty,\\
\sup_{t\in[0,T)}\norm{\omega(t)}{L^2}^2 &\leq \norm{B_0+\omega_0}{L^2}^2e^{CA(T)}+\norm{u_0}{L^2}^2+\norm{B_0}{L^2}^2.
\end{aligned}
\]
\end{enumerate} 
\end{remark}

A natural question that arises from Theorem \ref{Theorem 1.1} is whether or not $U_T$ in \eqref{eq:regularity-cond} can be bounded uniformly in time. There are two approaches to answer this question: (1) smallness condition of $B_0+\omega_0$; (2) the dimension reduction. We will deal with the second case in Section \ref{sec:1.1.2} (Theorem \ref{Theorem 1.2}). We now deal with the first case. More precisely, we are finding initial data with small $B_{0}+\omega_{0}$ from which \eqref{eq:bound-1}, \eqref{eq:bound-2} and \eqref{eq:bound-3} hold until $t=T^*$ when {$\frac{|\nu-\eta|}{\nu+\eta}$ is sufficiently small.}

\begin{corollary} \label{Corollary 1.1} \upshape
Let $\nu,\eta>0$ satisfy $\frac{|\nu-\eta|}{\nu+\eta}\leq \frac{1}{8}$ and let $(u_{0},B_{0})\in H^{3}(\mathbb{R}^3)$ with $\dv u_0=\dv B_0=0$. Let $(u,B)\in C([0,T^*);H^{3}(\mathbb{R}^3))$ be the  solution of \eqref{H MHD}. If $(u_0,B_0)$ satisfies
\eqn \label{eq:cor-cond}
C_0=\left(\norm{B_0+\omega_0}{L^2}^{2}+C\mathcal{E}_0\frac{(\nu-\eta)^2}{(\nu+\eta)^2}\right) \exp\left(\frac{C(\mathcal{E}_0+\mathcal{E}_0^2)}{(\nu+\eta)^{4}}\right)<1,
\een
then 
\eqn \label{eq:cor:bound}
\sup_{t\in[0,T^*)}\norm{(B+\omega)(t)}{L^2}^2+\frac{\nu+\eta}{2}\int^{T^*}_0\norm{\nabla(B+\omega)(t)}{L^2}^2\,dt \leq C_0.
\een
\end{corollary}

As one can see in Remark \ref{Remark 1.1} (2), $U_{T^{\ast}}$ is uniformly bounded by \eqref{eq:energy-ineq} and \eqref{eq:cor:bound}
\[
\begin{split}
U_{T^*} \leq C \int^{T^*}_0 \norm{\nabla u(t)}{L^2} \norm{\nabla \omega(t)}{L^2}\,dt \leq C\left(\frac{\mathcal{E}_0}{\nu}\right)^{\frac{1}{2}}\left(\frac{C_0}{\nu+\eta}\right)^{\frac{1}{2}}.
\end{split}
\]

\begin{remark} \label{Remark 1.2} \upshape
\noindent
\begin{enumerate}
\item In many electrically conducting fluids of physical interest, both the Reynolds number $Re=\nu^{-1}$ and the magnetic Reynolds number $Rm=\eta^{-1}$ can be large and comparable. {So in this case, we see that 
\[
\frac{|\nu-\eta|}{\nu+\eta}=\frac{|Re-Rm|}{Re+Rm}
\]
is very small and thus our assumption  is reasonable. This assumption is also imposed to the usual MHD (without the Hall term in (\ref{H MHD})) in \cite{He Huang Wang}.}
\item When $\nu=\eta>0$, \eqref{eq:cor-cond} is reduced to 
\eqn \label{eq:cor-cond2}
C_0=\norm{B_0+\omega_0}{L^2}^{2} \exp\left(\frac{C}{\nu^{4}}(\mathcal{E}_0+\mathcal{E}_0^2)\right)<1.
\een 
If $(u_0,B_0,\omega_0)\in L^2$ with \eqref{eq:cor-cond2}, we can find a global-in-time weak solution of (\ref{H MHD}) satisfying \eqref{eq:cor:bound}. Moreover, we can derive temporal decay rates of $B+\omega$ using the refined Fourier splitting method introduced in \cite{Bae Jung Shin}. However, the argument of deriving such decay rates of $B+\omega$ is rather different from those in Section \ref{sec:1.1} and Section \ref{sec:1.2} and thus we will deal with decay rates of $B+\omega$ in Section \ref{sec:5}.
\item Going a little further with $\nu=\eta>0$, if we assume $B_0+\omega_0\equiv0$, then $B+\omega\equiv0$ for all $t\geq0$. Then, we can write \eqref{H MHD} only in terms of $u$ or $B$: for example, $u$ satisfies 
\[
\partial_{t}u-\nu\Delta u-((u-\Delta u)\times\omega)+\nabla\left(p+|u|^2/2\right)=0,\quad \dv u=0.
\]
Equivalently, the equations of $B$ is written as
\eqn \label{decoupled eq of B}
\partial_t B-\nu\Delta B+\nabla \times ((A+\nabla\times B)\times B)=0,\quad  \nabla \times A= B,\quad \dv  A=0
\een
which is the electron MHD equations if we neglect $\nabla \times  (A\times B)$. We remark that the global well-posedness problems of \eqref{decoupled eq of B} and the electron MHD equations are open as well.
\end{enumerate}
\end{remark}

%%%%%%%%%%%%%%%%%%%%
\subsubsection{\bf Regularity of $B+\omega$ with 2D variables} \label{sec:1.1.2}
%%%%%%%%%%%%%%%%%%%
The Hall MHD with 2D variables, which is called the $2\frac{1}{2}$D Hall MHD, is described by $u=(u^{1}, u^{2}, u^{3}):[0,\infty)\times\mathbb{R}^{2}\rightarrow\mathbb{R}^{3}$ and $B=(B^{1}, B^{2}, B^{3}):[0,\infty)\times\mathbb{R}^{2}\rightarrow\mathbb{R}^{3}$ satisfying (\ref{H MHD}). To write (\ref{H MHD}) with the 2D variables, let $\widetilde{u}=(u^{1}, u^{2})$, $\widetilde{B}=(B^{1}, B^{2})$, and denote $\nabla=(\partial_{1}, \partial_{2})$, $\Delta=\partial^{2}_{1}+\partial^{2}_{2}$. Then (\ref{H MHD}) is reformulated as  
\begin{subequations}\label{H MHD New}
\begin{align}
&\partial_{t}u-\nu{\Delta} u+\widetilde{u}\cdot {\nabla} u-\widetilde{B}\cdot {\nabla} B+{\nabla} p=0, \label{H MHD New a}\\
&\partial_{t} \widetilde{B}-\eta {\Delta} \widetilde{B}+\widetilde{u}\cdot {\nabla} \widetilde{B}-\widetilde{B}\cdot {\nabla} \widetilde{u}+\widetilde{B}\cdot {\nabla} \widetilde{J}-\widetilde{J}\cdot {\nabla} \widetilde{B}=0,  \label{H MHD New b}\\ 
&\partial_{t} B^3-\eta {\Delta} B^3+\widetilde{u}\cdot {\nabla} B^3-\widetilde{B}\cdot {\nabla} u^3+\widetilde{B}\cdot {\nabla} J^3=0,  \label{H MHD New c}\\ 
&\dv u=0,\quad \dv B=0,
\end{align}
\end{subequations}
where $\widetilde{J}=(\partial_{2}B^{3}, -\partial_{1}B^{3})$ and $J^3=\partial_1 B^2-\partial_2 B^1$.  Compared to the $2\frac{1}{2}$D incompressible Navier-Stokes equations where the equations of $\widetilde{u}$ and $u^3$ are decoupled and so globally well-posed with smooth initial data, the same kind of decoupling  does not happen to (\ref{H MHD New}) due to the Hall term. Thus, the global well-posedness of the $2\frac{1}{2}$D Hall MHD is still a  challenging problem.

Parallel to the three dimensional case, we first deal with $B+\omega$.  We see that $U_T$ in \eqref{eq:regularity-cond} is bounded uniformly in time: by \eqref{eq:energy-ineq} and (\ref{BMO embedding}) with $d=2$
\[
\int^{T}_{0}\norm{u(t)}{\text{BMO}}^2\,dt\leq C\int^{T}_{0}\norm{\nabla u(t)}{L^2}^2\,dt\leq \frac{C\mathcal{E}_0}{\nu}
\]
for all $T\geq0$. Hence, we can state the following theorem without an additional assumption on $u$.

\begin{theorem}\label{Theorem 1.2}\upshape
Let $\nu,\eta>0$ and $(u_{0},B_{0})\in H^{2}(\mathbb{R}^2)$ with $\dv u_0=\dv  B_0=0$. There exists a unique solution $(u,B)\in C([0,T^*);H^2(\mathbb{R}^2))$ of \eqref{H MHD New} which has the following uniform-in-time bounds:
\[
\begin{split}
&\sup_{t\in [0,T^*)}\norm{(B+\omega)(t)}{L^2}^{2}+\nu\int^{T^*}_{0}\norm{\nabla(B+\omega)(t)}{L^{2}}^{2}\,dt\\
&\leq \left(\norm{B_{0}+\omega_{0}}{L^2}^{2}+C\mathcal{E}_0\frac{(\nu-\eta)^{2}}{\nu\eta}\right)\exp\left(\frac{C\mathcal{E}_0}{\nu^{2}}\right)= \mathcal{E}_1, 
\end{split}
\]
and
\eqn \label{eq:2.5D-bound-2}
\sup_{t\in[0,T^*)}\left\|\omega(t)\right\|^{2}_{L^{2}}+\min{(\nu,\eta)}\int^{T^*}_{0}\left\|\nabla\omega(t)\right\|^{2}_{L^{2}}\,dt\leq \mathcal{E}_{1}+\mathcal{E}_0=\mathcal{E}_2. 
\een
Moreover, if $\nu=\eta$ and $B_0+\omega_0\in H^k(\mathbb{R}^2)$ for each $k=1,2$, $B+\omega$ also satisfies
\begin{equation}\label{eq:2.5D-bound-3}
\sup_{t\in[0,T^*)}\left\|(B+\omega)(t)\right\|^{2}_{H^{k}}+\nu\int^{T^*}_{0}\left\|\nabla(B+\omega)(t)\right\|^{2}_{H^{k}}\,dt\leq \norm{B_{0}+\omega_{0}}{H^k}^{2}\exp\left(\frac{C\mathcal{E}_2}{\nu^2}\right).
\end{equation}
\end{theorem}

\vspace{1ex}

As an application of Theorem \ref{Theorem 1.2}, we provide a blow-up criterion of (\ref{H MHD New}). Before that,  we give some blow-up criteria  of \eqref{H MHD New}. Let $(u,B)\in C([0,T^*);H^2 (\mathbb{R}^2))$ be a solution of \eqref{H MHD New}. Then, the maximal existence time $T^*<\infty$ if and only if 
\[
\begin{split}
\text{\cite{Chae Lee}}: &\ \int^{T^*}_{0}\norm{\nabla \times B(t)}{\text{BMO}}^{2}\,dt=\infty,\\
\text{\cite{Bae Kang, Rahman}}: & \ \int^{T^*}_{0}\norm{\nabla B^{3}(t)}{L^{r'}}^{r}\,dt=\infty \ \text{or} \ \int^{T^*}_{0}\norm{\nabla \widetilde{B}(t)}{L^{r'}}^{r}\,dt=\infty, \quad \frac{1}{r}+\frac{1}{r'}=\frac{1}{2},\ 2\leq r<\infty.
\end{split}
\]

We now provide new blow-up criteria with the aid of Theorem \ref{Theorem 1.2}. To do so, we use $L^{\infty}([0,T];L^{2})\cap L^{2}([0,T];\dot{H}^{1}) \subset X^{r}_{T}$ for all $r\in[2,\infty)$, where a function space $X^{r}_{T}$ is given by the following norm:
\[
\norm{f}{X^{r}_{T}}^r:=
	\begin{cases}	\displaystyle \int^{T}_{0}\norm{f(t)}{L^{r'}}^{r}\,dt\quad\text{when}\ \ 2<r<\infty,\ \frac{1}{r}+\frac{1}{r'}=\frac{1}{2}, \\
	\displaystyle \int^{T}_{0}\norm{f(t)}{\dot{H}^{1}}^{2}\,dt\quad\text{when}\ \ r=2.
	\end{cases}	
\]
When $\nu=\eta>0$,  we deduce that $\nabla (B+\omega)\in X^{r}_{T}$ from \eqref{eq:2.5D-bound-3} with $k=1$. Then, both $\norm{\nabla B}{X^{r}_{T}}$ and $\norm{\nabla\omega}{X^{r}_{T}}$ are bounded for all $T>0$ or blow up at the same time $T^*<\infty$. From this observation, we are able to state a blow-up criterion  in terms of any single component $B^3$, $\widetilde{B}$, $\omega^3=\partial_1 u^2-\partial_2 u^1$, and $\widetilde{\omega}=(\partial_2 u^3,-\partial_1 u^3)$.

\begin{corollary} \label{Corollary 1.2}\upshape
Let $\nu,\eta>0$ and $(u_{0},B_{0})\in H^{2}(\mathbb{R}^2)$ with $\dv {u}_0=\dv {B}_0=0$. Let $T^*$ be the maximal existence time of the solution in Theorem \ref{Theorem 1.2}. Then,  $T^*<\infty$ is equivalent to each of the followings:
\eqn \label{BW 1}
\text{(i)}\ \left\|\nabla B^{3}\right\|_{X^{r}_{T^*}}=\infty, \quad  \text{(ii)} \ \left\|\nabla \widetilde{B}\right\|_{X^{r}_{T^*}}=\infty. 
\een
In addition, if $\nu=\eta$, then $T^*<\infty$ is equivalent to each of the followings as well:
\eqn \label{BW 2}
\text{(iii)} \ \left\|\nabla \omega^3\right\|_{X^{r}_{T^*}}=\infty,\quad \text{(iv)} \ \left\|\nabla \widetilde{\omega}\right\|_{X^{r}_{T^*}}=\infty.
\een
\end{corollary}

\vspace{1ex}

The conditions (\ref{BW 1}) for $r>2$ are already proved in \cite{Bae Kang, Rahman}, and the case $r=2$ is also easily deduced from the computations in \cite{Bae Kang, Rahman}. On the other hand, by Theorem \ref{Theorem 1.2}, 
\[
\left\|\nabla(B+\omega)\right\|^{r}_{X^{r}_{T^*}}=\left\|\nabla(B^{3}+\omega^3)\right\|^{r}_{X^{r}_{T^*}}+\left\|\nabla(\widetilde{B} +\widetilde{\omega})\right\|^{r}_{X^{r}_{T^*}}<\infty
\]
and so we prove \eqref{BW 2}.

\vspace{1ex}

Compared to \cite{Rahman} where the blow-up criteria are established in terms of $ \norm{\nabla\Delta\widetilde{\omega}}{X^r_{T^*}}$ when $\nu=\eta>0$, we remove two derivatives from the result in \cite{Rahman} by using Theorem \ref{Theorem 1.2}. Corollary \ref{Corollary 1.2} indicates that we can show the finite time blow-up or prove the existence of global-in-time solutions of \eqref{H MHD New} by investigating any single component of $(u,B)$.

%%%%%%%%%%%%%%%%%%%%%%%
\subsection{Global existence of the $2\frac{1}{2}$D Hall MHD} \label{sec:1.2}
%%%%%%%%%%%%%%%%%%%%%%%
In what follows, we discuss how to establish  the global existence of \eqref{H MHD New} with small initial data. There are some results of global existence of \eqref{H MHD New} with smallness assumption to initial data: $\norm{u_0}{H^1}+\norm{B_0}{H^1}$ when $\nu,\eta>0$ \cite{Bae Kang}; $\|B_{0}\|_{H^{1}}$ when $\nu =\eta$ \cite{Danchin Tan 3} and $\nu \ne \eta$ \cite{Tan}. Compared to \cite{Tan} where the global well-posedness of \eqref{H MHD New} with small $\norm{B_0}{H^1}$ is established by introducing a new functional, we can show the global well-posedness result with small $\norm{\nabla\times B_0}{L^2}$ when $\nu,\eta>0$ by directly applying the energy method and the regularity gain of $u$ from \eqref{eq:2.5D-bound-2}. We also note that $\dot{H}^1$ is a scaling invariant space of $B_0$ if we neglect the effect of $u$ from \eqref{H MHD New}.  But, we skip the proof of the global well-posedness result with small $\norm{\nabla\times B_0}{L^2}$ because we are able to prove the global well-posedness of \eqref{H MHD New} under a smallness assumption to only the third component of $\nabla\times B_{0}$.

To demonstrate our idea, we first deal with the $2\frac{1}{2}$D electron MHD: by removing $u$ in \eqref{H MHD New}
\eqn \label{2D EMHD}
\partial_{t} \widetilde{B}-\eta {\Delta} \widetilde{B}+\widetilde{B}\cdot {\nabla} \widetilde{J}-\widetilde{J}\cdot {\nabla} \widetilde{B}=0, \quad \partial_{t} B^3-\eta {\Delta} B^3+\widetilde{B}\cdot {\nabla} J^3=0.
\een

\begin{theorem} \label{Theorem 1.3}\upshape
Let $\eta>0$ and $B_0\in H^{2}(\mathbb{R}^2)$ with $\dv B_0=0$ and let
\eqn \label{constants EMHD}
\mathcal{H}_{0}=\norm{B_0}{H^2}^2 e^{C\eta^{-4}\norm{\nabla B_0}{L^2}^4}.
\een
There exists $\epsilon>0$ such that if 
\begin{equation}\label{eq:EMHD-data}
\norm{(\nabla\times B_0)^3}{L^2}^{2\exp{(-C\eta^{-2}\norm{\nabla B_0}{L^2}^2)}}\exp\left[\frac{C}{\eta^{2}}(1+\ln \mathcal{H}_{0})\norm{\nabla B_0}{L^2}^2\right] \leq \epsilon \eta^{2},
\end{equation} 
then there exists a unique global-in-time solution $B\in C([0,\infty);H^{2}(\mathbb{R}^2))$ of \eqref{2D EMHD} satisfying 
\[
\begin{split}
&\sup_{t\in[0,\infty)}\norm{B(t)}{H^2}^2+\eta\int^{\infty}_{0}\norm{\nabla B(t)}{H^2}^2\,dt\leq \mathcal{H}_{0},\\
& \sup_{t\in[0,\infty)}\norm{(\nabla\times B)^3(t)}{L^2}^{2}\leq \norm{(\nabla\times B_0)^3}{L^2}^{2\exp{(-C\eta^{-2}\norm{\nabla B_0}{L^2}^2)}}\exp\left[\frac{C}{\eta^{2}}(1+\ln \mathcal{H}_{0})\norm{\nabla B_0}{L^2}^2\right].
\end{split}
\]
\end{theorem}

\begin{remark}\upshape
\noindent
\begin{enumerate}
\item \eqref{eq:EMHD-data} is equivalent to 
\[
\norm{(\nabla\times B_0)^3}{L^2}^{2} \leq \left(\epsilon \eta^{2}\exp\left[-\frac{C}{\eta^{2}}(1+\ln \mathcal{H}_{0})\norm{\nabla B_0}{L^2}^2\right]\right)^{\exp{(C\eta^{-2}\norm{\nabla B_0}{L^2}^2)}}
\]
which allows initial data with large $\norm{\nabla B_0}{L^2}$ as long as $\norm{(\nabla\times B_0)^3}{L^2}$ is sufficiently small. \eqref{eq:EMHD-data} also allows any large $B^3_0$ when $\widetilde{B}_0\equiv0$ (so $\norm{(\nabla\times B_0)^3}{L^2}=0$): it  is because \eqref{2D EMHD} reduces to the heat equation $\partial_t B^3-\eta\Delta B^3=0$ when $\widetilde{B}\equiv0$.
\item {The first and second derivatives of $B$ decay in time:
\[
\norm{\nabla B(t)}{L^2}^2\leq \frac{C}{1+t},\quad \norm{\Delta B(t)}{L^2}^2\leq \frac{C}{(1+t)^2}
\]
which are classical results} following from \eqref{eq:4.8 dd} and \eqref{eq:4.10 dd} with the Sobolev interpolation. 
\item In the proof of Theorem \ref{Theorem 1.3} and Theorem \ref{Theorem 1.4} below, we crucially use Lemma \ref{Lemma 2.1} to set the smallness condition only  in terms of $(\nabla\times B_0)^3$ as detailed in Remark \ref{remark 4.1}.
\end{enumerate}
\end{remark}

We now present a similar result to the $2\frac{1}{2}$D Hall MHD.

\begin{theorem} \label{Theorem 1.4}\upshape
Let $\nu,\eta>0$ and $(u_0,B_0)\in H^{2}(\mathbb{R}^2)$ with $\dv u_0=\dv B_0=0$. Let $\mathcal{E}_0$ and $\mathcal{E}_2$ be defined in \eqref{eq:energy-ineq} and \eqref{eq:2.5D-bound-2}. We also define
\[
\begin{split}
\mathcal{H}_1&=\left(\norm{\nabla B_0}{L^2}^2+\frac{C\mathcal{E}_0\mathcal{E}_2}{\eta\min{(\nu,\eta)}}\right)\exp{\left(\frac{C\mathcal{E}_0}{\nu\eta}\right)},\qquad \mathcal{S}_1= \frac{\mathcal{H}_1}{\eta^2}+\frac{\mathcal{E}_0}{\nu\eta},\\
\mathcal{H}_2&= \left(\norm{B_0}{H^2}^2+\frac{C\mathcal{E}_0\mathcal{E}_2}{\eta\min{(\nu,\eta)}}+\frac{C\mathcal{E}_0\mathcal{E}_2^2}{\eta^4}\right)\exp{\left(\frac{C\mathcal{E}_2}{\eta\min{(\nu,\eta)}}+\frac{C\mathcal{H}_1^2}{\eta^4}\right)}+\norm{u_0}{L^2}^2.
\end{split}
\]
There exists $\epsilon>0$ such that if 
\[
\norm{(\nabla\times B_0)^3}{L^2}^{2\exp{(-C\mathcal{S}_1)}}e^{C\mathcal{S}_1(1+\ln \mathcal{H}_{2})} \leq \epsilon\eta^{2},
\]
there exists a unique global-in-time solution $(u,B)\in C([0,\infty);H^{2}(\mathbb{R}^2))$ of \eqref{H MHD New} satisfying 
\[
\begin{split}
&\sup_{t\in[0,\infty)}\norm{B(t)}{H^2}^2+\eta\int^{\infty}_{0}\norm{\nabla B(t)}{H^2}^2\,dt\leq \mathcal{H}_2,\\
&\sup_{t\in[0,\infty)}\norm{(\nabla\times B)^3(t)}{L^2}^{2}\leq \norm{(\nabla\times B_0)^3}{L^2}^{2\exp{(-C\mathcal{S}_1)}}e^{C\mathcal{S}_1(1+\ln \mathcal{H}_{2})}.
\end{split}
\]
\end{theorem}

\vspace{1ex}

Compared to Theorem \ref{Theorem 1.3}, Theorem \ref{Theorem 1.4} is rather complicated, but we prove  it as the proof of Theorem \ref{Theorem 1.3} because the regularity gain of $u$ from \eqref{eq:2.5D-bound-2} enables us to consider \eqref{H MHD New} as a perturbation of \eqref{2D EMHD}.
To the best of our knowledge, Theorem \ref{Theorem 1.3} and Theorem \ref{Theorem 1.4} are the first global well-posedness results with non-small magnetic field for the $2\frac{1}{2}$D electron MHD and Hall MHD.

%%%%%%%%%%%%%%%%
\section{Preliminaries}
%%%%%%%%%%%%%%%%
All generic constants will be denoted by $C$. We follow the convention that such constants can vary from expression to expression and even between two occurrences within the same expression. We  use the simplified form of the integral of the spatial variables:
\[
\int=:\int_{\mathbb{R}^{d}}dx, \quad d=2,3.
\]

For $i\in \mathbb{N}$, we denote $\nabla^{i}f$ by the full $i$-th order derivative of $f$. For $k\in \mathbb{N}$, $H^{k}$ is a energy space equipped with the following norm
\[
\|f\|^{2}_{H^{k}}=\|f\|^{2}_{L^{2}}+\sum^{k}_{i=1}\left\|\nabla^{i} f\right\|^{2}_{L^{2}}.
\]

Let $\widehat{f}$ and $\mathcal{F}(f)$ be the Fourier transform:
\[
\widehat{f}(\xi)=\mathcal{F}(f)(\xi)=\int f(x)e^{-ix\cdot \xi}dx.
\]
By using the Fourier transform, we can also define $H^{s}$, $s\notin \mathbb{N}$:
\[
\|f\|^{2}_{H^{s}}=\int (1+|\xi|^{2})^{s}\left|\widehat{f}(\xi)\right|^{2}, \quad \|f\|^{2}_{\dot{H}^{s}}=\int |\xi|^{2s}|\widehat{f}(\xi)|^2.
\]

We present some Sobolev estimates. For $1<p<\infty$ and a smooth divergence-free vector field $f$,
\begin{equation}\label{curl-Lp}
\norm{\nabla f}{L^{p}}\leq C(p) \norm{\nabla\times f}{L^{p}},
\end{equation}
We also use the following inequalities: 
\begin{subequations} \label{Sobolev inequalties}
\begin{align}
\text{3D}: \ &\left\|f\right\|_{L^{3}}\leq C \left\|f\right\|_{\dot{H}^{\frac{1}{2}}}, \ \ \left\|f\right\|_{L^{6}}\leq C \left\|f\right\|_{\dot{H}^{1}}, \ \ \|f\|_{L^{\infty}} \leq C \|\nabla f\|^{1/2}_{L^{2}} \|\Delta f\|^{1/2}_{L^{2}} \label{Sobolev inequalties a}\\
\text{2D}: \ & \norm{f}{L^{4}}\leq C\norm{f}{L^{2}}^{1/2}\norm{\nabla f}{L^{2}}^{1/2}\\
\text{Interpolation}: \  &\|f\|_{\dot{H^{s}}}\leq \|f\|^{\theta}_{\dot{H^{s_{0}}}}\|f\|^{1-\theta}_{\dot{H^{s_{1}}}}, \quad s_{0}<s<s_{1}, \ s=\theta s_{0}+(1-\theta)s_{1}. \label{Interpolation}
\end{align}
\end{subequations}

We now recall some properties of $\text{BMO}$ and the Hardy space $\mathcal{H}$: see \cite[Chapter 6]{Duoandikoetxea} for the definitions  of these two spaces.

\begin{enumerate}
\item  \cite[Theorem II.1]{Coifman}: if $(E,B)$ satisfy $\dv E=0$ and $\curl B=0$,  then $E\cdot B\in \mathcal{H}$ and
\eqn \label{Hardy 1}
\left\|E\cdot B\right\|_{\mathcal{H}}\leq C \|E\|_{L^{2}} \|B\|_{L^{2}}.
\een
\item Duality of $\text{BMO}$ and $\mathcal{H}$ \cite{Fefferman}:
\eqn \label{BMO 1}
\Big|\int fg \Big|\leq C \|f\|_{\text{BMO}} \|g\|_{\mathcal{H}}.
\een
\item Embedding property of $\text{BMO}$ \cite[Section 1.2]{Brezis Nirenberg}: $\dot{H}^{\frac{d}{2}}\subset \text{BMO}$ and 
\eqn \label{BMO embedding}
\|f\|_{\text{BMO}}\leq C \|f\|_{\dot{H}^{s}}, \quad s=\frac{d}{2}.
\een
\end{enumerate}

We also recall Gr\"onwall's inequality \cite[Page 624]{Evans}: if $f\geq0$ is an absolutely continuous function satisfying $
f'(t) \leq \phi(t)f(t) +\psi(t)$ for any non-negative and integrable functions $\phi$ and $\psi$, then 
\[
f(t) \leq \left(f(0)+\int^{t}_{0}\psi(\tau)d\tau\right)\exp\int^{t}_{0}\phi(\tau)\,d\tau \quad \text{for all $t\leq T$}.
\]

We finally give the following lemma which will be used in the proof of Corollary \ref{Corollary 1.1}.

\begin{lemma}[{\cite[Modified version of Lemma A.4]{Danchin Tan 3}}]\label{Lemma 2.2}\upshape
Let $D, E, W$ be non-negative measurable functions and $X$ be a non-negative differentiable function. Suppose there exists $\alpha,\beta\geq0$ such that
\[
\frac{d}{dt}X+\left(1+\beta(1-X^\alpha)\right) D \leq E+ WX
\]
on $0\leq t\leq T$. If 
\[
\left( X(0)+\int^T_0 E(t)\,dt\right)\exp{\left(\int^{T}_{0}W(t)\,dt\right)}<1,
\]
then, for $0\leq t\leq T$,
\[
X(t)+\int^{t}_{0}D(\tau)\,d\tau\leq \left( X(0)+\int^t_0 E(\tau)\,d\tau\right)\exp{\left(\int^{t}_{0}W(\tau)\,d\tau\right)}.
\]
\end{lemma}

%%%%%%%%%%%%%%%%%%%%%
\section{Proof of Theorem \ref{Theorem 1.1} and Corollary \ref{Corollary 1.1}}
%%%%%%%%%%%%%%%%%%%%
In this section, we prove Theorem \ref{Theorem 1.1} and Corollary \ref{Corollary 1.1}. 

%%%%%%%%%%%%%%%%%%%%%%%%%%%
\subsection{Proof of Theorem \ref{Theorem 1.1}}
%%%%%%%%%%%%%%%%%%%%%%%%%%%
From \eqref{eq:B+omega} with (\ref{Hardy 1}) and (\ref{BMO 1}), it follows that 
\begin{align*}
&\frac{1}{2}\frac{d}{dt}\norm{B+\omega}{L^{2}}^{2}+\nu\norm{\nabla(B+\omega)}{L^{2}}^{2}=\int ((B+\omega)\cdot\nabla u)\cdot (B+\omega)+(\eta-\nu)\int\Delta B \cdot (B+\omega)\\
&=-\int ((B+\omega)\cdot\nabla(B+\omega))\cdot u+(\nu-\eta)\int\nabla B :\nabla(B+\omega)\\
&\leq \norm{u}{\text{BMO}}\norm{(B+\omega)\cdot\nabla(B+\omega)}{\mathcal{H}}+|\nu-\eta|\norm{\nabla B}{L^2}\norm{\nabla(B+\omega)}{L^2}\\
&\leq \norm{u}{\text{BMO}}\norm{B+\omega}{L^2}\norm{\nabla(B+\omega)}{L^2}+|\nu-\eta|\norm{\nabla B}{L^2}\norm{\nabla(B+\omega)}{L^2}\\
&\leq \frac{\nu}{2}\norm{\nabla(B+\omega)}{L^2}^{2}+\frac{C}{\nu}\norm{u}{\text{BMO}}^{2}\norm{B+\omega}{L^2}^{2}+\frac{C}{\nu}(\nu-\eta)^{2}\norm{\nabla B}{L^2}^{2}.
\end{align*}
From this, we deduce that 
\eqn \label{B+omega L2}
\frac{d}{dt}\norm{B+\omega}{L^{2}}^{2}+\nu\norm{\nabla(B+\omega)}{L^{2}}^{2}\leq \frac{C}{\nu}\norm{u}{\text{BMO}}^{2}\norm{B+\omega}{L^2}^{2}+\frac{C(\nu-\eta)^{2}}{\nu}\norm{\nabla B}{L^2}^{2}.
\een
By Gr\"onwall's inequality with  \eqref{eq:energy-ineq}, we derive \eqref{eq:bound-1}. Moreover, \eqref{eq:bound-2} is proved by using \eqref{eq:bound-1} and \eqref{eq:energy-ineq}:
\[
\begin{split}
& \sup_{t\in[0,T)}\left\|\omega(t)\right\|^{2}_{L^{2}}+\min{(\nu,\eta)}\int^{T}_{0}\left\|\nabla\omega(t)\right\|^{2}_{L^{2}}dt \\
&\leq \sup_{t\in[0,T)}\left\|(B+\omega)(t)\right\|^{2}_{L^{2}}+\nu\int^{T}_{0}\left\|\nabla (B+\omega)(t)\right\|^{2}_{L^{2}}dt +\sup_{t\in[0,T)}\left\|B(t)\right\|^{2}_{L^{2}} +\eta\int^{T}_{0}\left\|\nabla B(t)\right\|^{2}_{L^{2}}dt \\
&\leq \mathcal{E}_{1}(U_T)+\left\|u_{0}\right\|^{2}_{L^{2}} +\left\|B_{0}\right\|^{2}_{L^{2}}.
\end{split}
\]

We now assume that $\nu=\eta>0$. Then \eqref{eq:B+omega} is reduced to
\eqn \label{eq:B+omega2}
\begin{split}
&\partial_{t}(B+\omega)-\nu\Delta(B+\omega)+u\cdot\nabla (B+\omega)-(B+\omega)\cdot\nabla u=0,\\
& \dv(B+\omega)=\dv u=0.
\end{split}
\een
By taking the $L^2$ inner product of \eqref{eq:B+omega2} with $-\Delta(B+\omega)$,  we have 
\begin{align*}
&\frac{1}{2}\frac{d}{dt}\norm{\nabla(B+\omega)}{L^2}^{2}+\nu\norm{\Delta(B+\omega)}{L^2}^{2}\\
&=\int (u\cdot\nabla (B+\omega))\cdot \Delta(B+\omega)-\int ((B+\omega)\cdot\nabla u)\cdot  \Delta(B+\omega)\\
&=-\int \left(\partial_{k}u\cdot\nabla(B+\omega)\right)\cdot \partial_{k}(B+\omega)-\int \left((B+\omega)\cdot\nabla u\right)\cdot \Delta(B+\omega) \\
&\leq C\norm{\nabla u}{L^{3}}\norm{\nabla(B+\omega)}{L^2}\norm{\Delta(B+\omega)}{L^2} \leq \frac{\nu}{2}\norm{\Delta(B+\omega)}{L^2}^{2}+\frac{C}{\nu}\norm{\nabla u}{L^3}^2\norm{\nabla(B+\omega)}{L^2}^{2}
\end{align*}
and so we obtain
\eqn \label{B+omega H1}
\frac{d}{dt}\norm{\nabla(B+\omega)}{L^2}^{2}+\nu\norm{\Delta(B+\omega)}{L^2}^{2} \leq \frac{C}{\nu}\norm{\nabla u}{L^2}\norm{\Delta u}{L^2}\norm{\nabla(B+\omega)}{L^2}^{2}.
\een 
Similarly, the $L^2$ inner product of \eqref{eq:B+omega2} with $\Delta^2(B+\omega)$ gives 
\begin{align*}
\frac{1}{2}\frac{d}{dt}\norm{\Delta(B+\omega)}{L^2}^{2}&+\nu\norm{\nabla\Delta(B+\omega)}{L^2}^{2}=-\int\left[ \Delta u\cdot \nabla(B+\omega)+2\partial_k u\cdot\nabla \partial_k(B+\omega)\right]\cdot \Delta(B+\omega)\\
&-\int\nabla\left((B+\omega)\cdot\nabla u\right):\nabla\Delta(B+\omega)\\
&\leq C\left(\norm{B+\omega}{L^{\infty}}+\norm{\nabla(B+\omega)}{L^3}\right)\norm{\Delta u}{L^2}\norm{\nabla\Delta(B+\omega)}{L^2} \\
&\leq C\norm{\nabla(B+\omega)}{L^2}^{1/2}\norm{\Delta(B+\omega)}{L^2}^{1/2}\norm{\Delta u}{L^2}\norm{\nabla\Delta(B+\omega)}{L^2}\\
&\leq \frac{\nu}{2}\norm{\nabla\Delta(B+\omega)}{L^{2}}^{2}+\frac{C}{\nu}\norm{\Delta u}{L^{2}}^{2}\norm{\nabla(B+\omega)}{L^2}\norm{\Delta(B+\omega)}{L^2}.
\end{align*}
From this, we also obtain 
\eqn \label{B+omega H2}
\frac{d}{dt}\norm{\Delta(B+\omega)}{L^2}^{2}+\nu\norm{\nabla\Delta(B+\omega)}{L^2}^{2}
\leq \frac{C}{\nu}\norm{\Delta u}{L^{2}}^{2}\norm{\nabla(B+\omega)}{L^2}\norm{\Delta(B+\omega)}{L^2}.
\een
So \eqref{B+omega L2}, (\ref{B+omega H1}) and (\ref{B+omega H2}) yield  
\[
\frac{d}{dt}\norm{B+\omega}{H^2}^{2}+\nu\norm{\nabla(B+\omega)}{H^2}^{2} \leq \frac{C}{\nu}\left(\norm{\nabla u}{L^2}^2+\norm{\Delta u}{L^{2}}^{2}\right)\norm{B+\omega}{H^2}^2.
\]
By Gr\"onwall's inequality with \eqref{eq:energy-ineq} and \eqref{eq:bound-2}, we arrive at \eqref{eq:bound-3}.

%%%%%%%%%%%%%%%%%%%%%%%%%%%%%%%%
\subsection{Proof of Corollary \ref{Corollary 1.1}}\label{sec:3.2}
%%%%%%%%%%%%%%%%%%%%%%%%%%%%%%%%
First of all, we rewrite \eqref{eq:B+omega} as follows:
\begin{equation}\label{eq:B+omega3}
\begin{split}
&\partial_{t}(B+\omega)-\frac{\nu+\eta}{2} \Delta (B+\omega)+u\cdot \nabla (B+\omega) -(B+\omega)\cdot \nabla u=\frac{\nu-\eta}{2}\Delta(B+\omega)+(\eta-\nu)\Delta B,\\
& \dv (B+\omega)=0.
\end{split}
\end{equation}
By taking the $L^2$ inner product of \eqref{eq:B+omega3} with $B+\omega$, we have 
\begin{align*}
& \frac{1}{2}\frac{d}{dt}\norm{B+\omega}{L^2}^2+\frac{\nu+\eta}{2}\norm{\nabla(B+\omega)}{L^2}^2 \\
& =\int \left((B+\omega)\cdot\nabla u\right) \cdot (B+\omega)-\frac{\nu-\eta}{2}\norm{\nabla(B+\omega)}{L^2}^2+(\nu-\eta)\int\nabla B:\nabla (B+\omega)\\
&=(\text{I})+(\text{II})+(\text{III}).
\end{align*}
Since $\frac{|\nu-\eta|}{\nu+\eta}\leq \frac{1}{8}$, we bound 
\[
(\text{II})+(\text{III})\leq \frac{\nu+\eta}{8}\norm{\nabla(B+\omega)}{L^2}^2+ C\frac{(\nu-\eta)^2}{\nu+\eta}\norm{\nabla B}{L^2}^2.
\]
We also bound (I) as
\begin{align*}
(\text{I})& \leq \norm{B+\omega}{L^6}\norm{B+\omega}{L^3}\norm{\nabla u}{L^2}\leq C\norm{\nabla(B+\omega)}{L^2}^{3/2}\norm{B+\omega}{L^2}^{1/2}\norm{\nabla u}{L^2}^{1/2}\left(\norm{B+\omega}{L^2}^{1/2}+\norm{B}{L^2}^{1/2}\right)\\
&\leq \frac{\nu+\eta}{16}\norm{B+\omega}{L^2}^{2/3}\norm{\nabla(B+\omega)}{L^2}^{2}+\frac{\nu+\eta}{16}\norm{\nabla(B+\omega)}{L^2}^2 +\frac{C}{(\nu+\eta)^3}(1+\norm{B}{L^2}^2)\norm{\nabla u}{L^2}^2\norm{B+\omega}{L^2}^2.
\end{align*}
Thus, we obtain
\begin{align*}
& \frac{d}{dt}\norm{B+\omega}{L^2}^2+\frac{\nu+\eta}{2}\left(1+\frac{1-\norm{B+\omega}{L^2}^{2/3}}{4}\right)\norm{\nabla(B+\omega)}{L^2}^2 
\\
&\leq \frac{C}{(\nu+\eta)^3}(1+\norm{B}{L^2}^2)\norm{\nabla u}{L^2}^2\norm{B+\omega}{L^2}^2+C\frac{(\nu-\eta)^2}{\nu+\eta}\norm{\nabla B}{L^2}^2.
\end{align*}
By Lemma \ref{Lemma 2.2}, if $(u_0,B_0)$ satisfies 
\[
\begin{split}
&\left(\norm{B_0+\omega_0}{L^2}^{2}+C\frac{(\nu-\eta)^2}{\nu+\eta}\int^t_0\norm{\nabla B(\tau)}{L^2}^2\,d\tau\right) \exp\left[\frac{C}{(\nu+\eta)^3}\int^{t}_{0}(1+\norm{B(\tau)}{L^2}^2)\norm{\nabla u(\tau)}{L^2}^2 d\tau\right]\\
&\leq \left(\norm{B_0+\omega_0}{L^2}^{2}+C\mathcal{E}_0\frac{(\nu-\eta)^2}{(\nu+\eta)^2}\right) \exp\left(\frac{C(\mathcal{E}_0+\mathcal{E}_0^2)}{(\nu+\eta)^{4}}\right)<1,
\end{split}
\]
where we use \eqref{eq:energy-ineq} with the fact {that $\frac{\nu+\eta}{2}$, $\nu$ and $\eta$ are comparable to each other}. Hence, we arrive at
\[
\begin{split}
&\sup_{t\in[0,T^*)}\norm{(B+\omega)(t)}{L^2}^2+\frac{\nu+\eta}{2}\int^{T^*}_{0}\norm{\nabla(B+\omega)(t)}{L^2}^2\,dt \\
&\leq \left(\norm{B_0+\omega_0}{L^2}^{2}+C\mathcal{E}_0\frac{(\nu-\eta)^2}{(\nu+\eta)^2}\right) \exp\left(\frac{C(\mathcal{E}_0+\mathcal{E}_0^2)}{(\nu+\eta)^{4}}\right).
\end{split}
\]

%%%%%%%%%%%%%%%%%%%%%%%%%%%%%%%%%%%%%%%%
\section{Proof of Theorem \ref{Theorem 1.3} and Theorem \ref{Theorem 1.4}}
%%%%%%%%%%%%%%%%%%%%%%%%%%%%%%%%%%%%%%%%
To prove Theorem \ref{Theorem 1.3} and Theorem \ref{Theorem 1.4}, we need the following logarithmic Sobolev inequality.

\begin{lemma} \label{Lemma 2.1}\upshape
For $f\in H^{2}(\mathbb{R}^{2})$,
\[
\norm{f}{L^{\infty}}\leq C\norm{\nabla f}{L^2}\left[1+\left(\ln{\frac{\norm{f}{H^2}^2}{\norm{\nabla f}{L^2}^2}}\right)^{1/2}\right].
\]
\end{lemma}

\begin{proof}
Using $\norm{f}{L^\infty}\leq \norm{\widehat{f}}{L^1}$, we bound $\norm{f}{L^\infty}$ as follows:
\begin{align*}
\norm{f}{L^\infty} &\leq \int_{|\xi|\leq R^{-1}}|\widehat{f}(\xi)|\,d\xi+\int_{R^{-1}\leq |\xi|\leq R}|\widehat{f}(\xi)|\,d\xi+\int_{|\xi|\ge R}|\widehat{f}(\xi)|\,d\xi \\
&\leq \norm{f}{L^2}\Big(\int_{|\xi|\leq R^{-1}}\,d\xi\Big)^{1/2}+\norm{\nabla f}{L^2}\left(\int_{R^{-1}\leq |\xi|\leq R}|\xi|^{-2}\,d\xi\right)^{1/2}+\norm{\Delta f}{L^2}\left(\int_{|\xi|\ge R}|\xi|^{-4}\,d\xi\right)^{1/2}\\
&\leq C\norm{f}{H^2}/R+C\norm{\nabla f}{L^2}(\ln{R})^{\frac{1}{2}}
\end{align*}
for any $R\geq 1$. Lemma \ref{Lemma 2.1} follows by choosing $R=\norm{f}{H^2}/\norm{\nabla f}{L^2}$.
\end{proof}

%%%%%%%%%%%%%%%%%%%%%%%%
\subsection{Proof of Theorem \ref{Theorem 1.3}}
%%%%%%%%%%%%%%%%%%%%%%%%

We begin with the $2\frac{1}{2}$D electron  MHD:
\eqn \label{2.5 EHM}
\partial_{t}B-\eta \Delta B+\widetilde{\nabla}\times((\widetilde{\nabla}\times B)\times B)=0,\quad \dv B=0,\quad \widetilde{\nabla}=(\partial_1,\partial_2,0).
\een
Using $B=(\partial_2\psi,-\partial_1\psi,B^3)$, (\ref{2.5 EHM}) can be written as 
\begin{subequations} \label{2.5 EHM 2}
\begin{align}
&\partial_{t}\psi-\eta\Delta\psi=-\nabla B^3\cdot\nabla^{\perp}\psi,  \label{2.5 EHM 2 a}\\
& \partial_{t}B^{3}-\eta\Delta B^{3}=\nabla\Delta\psi\cdot\nabla^{\perp}\psi.  \label{2.5 EHM 2 b}
\end{align}
\end{subequations}
We note $(\nabla\times B)^3=\partial_1 B^2-\partial_2 B^1=-\Delta\psi$. To prove Theorem \ref{Theorem 1.3}, we use the following proposition.

\begin{proposition}\label{prop1}\upshape
Let $B\in C([0,T];H^{2})$ be the solution of \eqref{2.5 EHM}. There exists $\epsilon>0$ such that if 
\eqn \label{eq:5.2}
\sup_{t\in[0,T]}\norm{\Delta\psi(t)}{L^2}^2\leq {\epsilon}\eta^2,
\een
$B$ satisfies the following bounds: 
\begin{subequations} \label{eq:bound-B1 n}
\begin{align}
&\sup_{t\in[0,T]}\norm{\nabla B(t)}{L^2}^2+\eta\int^T_0\norm{\Delta B(t)}{L^2}^2\,dt\leq \norm{\nabla B_0}{L^2}^2, \label{eq:bound-B1} \\
&\sup_{t\in[0,T]}\norm{B(t)}{H^2}^2+\eta\int^{T}_{0}\norm{\nabla B(t)}{H^2}^2\,dt \leq \norm{B_0}{H^2}^2 e^{C\eta^{-4}\norm{\nabla B_0}{L^2}^4}= \mathcal{H}_{0}, \label{eq:bound-B1 n a}\\
& \sup_{t\in[0,T]}\norm{\Delta\psi(t)}{L^2}^{2} \leq \norm{\Delta\psi_{0}}{L^2}^{2\exp{(-C\eta^{-2}\norm{\nabla B_0}{L^2}^2)}}\exp\left[\frac{C}{\eta^{2}}(1+\ln \mathcal{H}_{0})\norm{\nabla B_0}{L^2}^2\right]. \label{eq:bound-B1 n b}
\end{align}
\end{subequations}
\end{proposition}

\begin{proof}
We first note the energy inequality:
\eqn\label{EMHD energy}
\left\|B(t)\right\|^{2}_{L^{2}}+2\eta\int^{t}_{0}\left\|\nabla B(\tau)\right\|^{2}_{L^{2}}d\tau\leq \left\|B_{0}\right\|^{2}_{L^{2}}.
\een
By taking the $L^2$ inner product of \eqref{2.5 EHM} with $-\Delta B$, we have 
\[
\begin{split}
\frac{1}{2}\frac{d}{dt}\norm{\nabla B}{L^2}^2+\eta\norm{\Delta B}{L^2}^2 &=\int\widetilde{\nabla}\times((\widetilde{\nabla}\times B)\times B)\cdot\Delta B \\
&=-\int\Delta(\nabla B^3\cdot\nabla^{\perp}\psi)\Delta\psi-\int(\nabla\Delta\psi\cdot\nabla^{\perp}\psi)\Delta B^3 =\text{(I)}.
\end{split}
\]
By applying the integration by parts two times, we see that 
\eqn \label{eq:Delta-psi}
\int \Delta(\nabla B^{3}\cdot\nabla^{\perp}\psi)\Delta\psi=-\int(\nabla\Delta \psi\cdot\nabla^{\perp} \psi)\Delta B^3+2\int(\nabla\partial_{k}B^3\cdot\nabla^{\perp}\partial_{k}\psi)\Delta\psi
\een
and thus $\text{(I)}$ is estimated as 
\eqn \label{eq:5.8}
\text{(I)}=-2\int(\nabla\partial_{k}B^3\cdot\nabla^{\perp}\partial_{k}\psi)\Delta\psi\leq C\norm{\Delta B^{3}}{L^{2}}\norm{\Delta \psi}{L^{4}}^{2}\leq C\norm{\Delta\psi}{L^2}\norm{\Delta B}{L^2}^2. 
\een
Therefore, if (\ref{eq:5.2}) holds with $2C\sqrt{\epsilon}\leq 1$, we obtain
\eqn \label{eq:4.8 dd}
\frac{d}{dt}\norm{\nabla B}{L^2}^2+\eta\norm{\Delta B}{L^2}^2\leq 0
\een
which implies \eqref{eq:bound-B1}.  By taking the $L^2$ inner product of (\ref{2.5 EHM}) with $\Delta^2 B$, we are led to 
\begin{equation}\label{eq:II3}
\begin{aligned}
\frac{1}{2}\frac{d}{dt}\norm{\Delta B}{L^2}^2&+\eta\norm{\nabla\Delta B}{L^2}^2 =-\int\Delta(J\times B)\cdot\Delta J \\
&= -\int(2\partial_{k}J\times\partial_{k}B+J\times\Delta B)\cdot\Delta J \leq C\norm{\nabla B}{L^4}\norm{\nabla^2B}{L^4}\norm{\Delta J}{L^2} \\
&\leq C\norm{\nabla B}{L^2}^{1/2}\norm{\Delta B}{L^2}\norm{\nabla\Delta B}{L^2}^{3/2} \leq \frac{\eta}{2}\norm{\nabla\Delta B}{L^2}^2+\frac{C}{\eta^{3}}\norm{\nabla B}{L^2}^2\norm{\Delta B}{L^2}^4
\end{aligned}
\end{equation}
and so we obtain
\eqn \label{eq:4.10 dd}
\frac{d}{dt}\norm{\Delta B}{L^2}^2+\eta\norm{\nabla\Delta B}{L^2}^2\leq \frac{C}{\eta^{3}}\norm{\nabla B}{L^2}^2\norm{\Delta B}{L^2}^2\norm{\Delta B}{L^2}^2.
\een
By Gr\"onwall's inequality and \eqref{eq:bound-B1}, we also derive 
\eqn \label{eq:bound-B2}
\sup_{t\in[0,T]}\norm{\Delta B(t)}{L^2}^2+\eta\int^T_0\norm{\nabla \Delta B(t)}{L^2}^2\,dt \leq \norm{\Delta B_0}{L^2}^2 e^{C\eta^{-4}\norm{\nabla B_0}{L^2}^4}.
\een
By (\ref{EMHD energy}), \eqref{eq:bound-B1}, and \eqref{eq:bound-B2}, (\ref{eq:bound-B1 n a}) follows.

\vspace{1ex}

 We now bound $\Delta \psi$.  By taking the $L^{2}$ inner product of \eqref{2.5 EHM 2 a} with $\Delta^2\psi$ and using \eqref{eq:Delta-psi}, we have
\begin{align*}
\frac{1}{2}\frac{d}{dt}\norm{\Delta\psi}{L^2}^2&+\eta\norm{\nabla\Delta\psi}{L^2}^2=\int(\nabla\Delta\psi\cdot\nabla^{\perp}\psi)\Delta B^3-2\int(\nabla\partial_k B^3\cdot\nabla^{\perp}\partial_k\psi)\Delta\psi\\
&\leq C\norm{\nabla\psi}{L^{\infty}}\norm{\Delta B^{3}}{L^2}\norm{\nabla\Delta\psi}{L^2}+C\norm{\Delta B^3}{L^2}\norm{\Delta\psi}{L^4}^2\\
&\leq C\left(\norm{\nabla\psi}{L^{\infty}}+\norm{\Delta\psi}{L^2}\right)\norm{\Delta B^{3}}{L^2}\norm{\nabla\Delta\psi}{L^2}\\
&\leq \frac{\eta}{2}\norm{\nabla\Delta\psi}{L^2}^2+\frac{C}{\eta}\left(\norm{\nabla\psi}{L^{\infty}}^2+\norm{\Delta\psi}{L^2}^2\right)\norm{\Delta B}{L^2}^2.
\end{align*}
From this, 
\eqn \label{psi with log}
\frac{d}{dt}\norm{\Delta\psi}{L^2}^2+\eta\norm{\nabla\Delta\psi}{L^2}^2\leq \frac{C}{\eta}\left(\norm{\nabla\psi}{L^{\infty}}^2+\norm{\Delta\psi}{L^2}^2\right)\norm{\Delta B}{L^2}^2.
\een
By Lemma \ref{Lemma 2.1} and (\ref{eq:bound-B1 n a}),
\begin{align*}
\frac{d}{dt}\norm{\Delta\psi}{L^2}^2+\eta\norm{\nabla\Delta\psi}{L^2}^2 &\leq \frac{C}{\eta}\left[1+\ln{\left(\frac{\norm{\nabla\psi}{H^2}^2}{\norm{\Delta \psi}{L^2}^2}\right)}\right]\norm{\Delta B}{L^2}^2\norm{\Delta\psi}{L^2}^2 \\
&\leq \frac{C}{\eta}\left(1+\ln{\mathcal{H}_0}+\ln{\norm{\Delta\psi}{L^2}^{-2}}\right)\norm{\Delta B}{L^2}^2\norm{\Delta\psi}{L^2}^2.
\end{align*}
Let  
\[
F(t)=\norm{\Delta\psi(t)}{L^2}^{2} \exp\left[-\frac{C}{\eta}(1+\ln{\mathcal{H}_0})\int^{t}_{0}\norm{\Delta B(\tau)}{L^2}^2\,d\tau\right].
\]
Then $F$ satisfies 
\begin{align*}
\frac{d}{dt}F &= \left(\frac{d}{dt}\norm{\Delta\psi}{L^2}^2-\frac{C}{\eta}(1+\ln{\mathcal{H}_0})\norm{\Delta B}{L^2}^2\norm{\Delta\psi}{L^2}^2\right)\exp\left[-\frac{C}{\eta}(1+\ln{\mathcal{H}_0})\int^{t}_{0}\norm{\Delta B(\tau)}{L^2}^2\,d\tau\right] \\
&\leq \frac{C}{\eta}\ln{\norm{\Delta\psi}{L^2}^{-2}}\norm{\Delta B}{L^2}^2\norm{\Delta\psi}{L^2}^2 \exp\left[-\frac{C}{\eta}(1+\ln{\mathcal{H}_0})\int^{t}_{0}\norm{\Delta B(\tau)}{L^2}^2\,d\tau\right] \\
&\leq \frac{C}{\eta}\norm{\Delta B}{L^2}^2F\ln{F^{-1}}.
\end{align*}
By solving this inequality, we deduce that 
\[
\frac{d}{dt}\left(\ln\ln  F^{-1}\right)\geq -\frac{C}{\eta}\norm{\Delta B}{L^2}^2.
\]
Integrating this in time and using \eqref{eq:bound-B1}, we bound $F$ as 
\[
F(t)\leq \exp\left[\ln F_{0} \exp\left(-\frac{C}{\eta^{2}} \norm{\nabla B_0}{L^2}^2\right)\right]\leq F_{0}^{\exp{(-C\eta^{-2}\norm{\nabla B_0}{L^2}^2)}}
\]
from this we find 
\[
\sup_{t\in[0,T]}\norm{\Delta\psi(t)}{L^2}^{2}\leq \norm{\Delta\psi_{0}}{L^2}^{2\exp{(-C\eta^{-2}\norm{\nabla B_0}{L^2}^2)}}\exp\left[\frac{C}{\eta^{2}}(1+\ln \mathcal{H}_{0})\norm{\nabla B_0}{L^2}^2\right].
\]
This completes the proof of Proposition \ref{prop1}.
\end{proof}

Theorem \ref{Theorem 1.3} follows by Proposition \ref{prop1} and the initial condition
\eqn \label{smallness from log}
\norm{\Delta\psi_{0}}{L^2}^{2\exp{(-C\eta^{-2}\norm{\nabla B_0}{L^2}^2)}}\exp\left[\frac{C}{\eta^{2}}(1+\ln \mathcal{H}_{0})\norm{\nabla B_0}{L^2}^2\right]\leq \epsilon \eta^{2}.
\een

\begin{remark}\upshape \label{remark 4.1}
In the proof of Theorem \ref{Theorem 1.3}. we crucially use Lemma \ref{Lemma 2.1} to derive the smallness condition  (\ref{smallness from log}). If we use $\left\|\nabla \psi\right\|^{2}_{L^{\infty}}\leq C\left\|\nabla \psi\right\|_{L^{2}}\left\|\nabla \Delta \psi\right\|_{L^{2}}$, instead of Lemma \ref{Lemma 2.1}, we can obtain 
\[
\frac{d}{dt}\norm{\Delta\psi}{L^2}^2+\eta\norm{\nabla\Delta\psi}{L^2}^2 \leq \frac{C}{\eta^{2}}\left\|\nabla \psi\right\|_{L^{2}}^2\norm{\Delta B}{L^2}^4 +\frac{C}{\eta}\norm{\Delta B}{L^2}^2\norm{\Delta\psi}{L^2}^2.
\]
Integrating this in time with using (\ref{eq:bound-B1}) and (\ref{eq:bound-B2})
\eqn \label{psi without log}
\begin{split}
&\norm{\Delta \psi(t)}{L^2}^2+\eta\int^{t}_{0}\norm{\nabla\Delta\psi(\tau)}{L^2}^2d\tau \\
&\leq  \left(\norm{\Delta\psi_{0}}{L^2}^2+\frac{C}{\eta^{3}}\left\|B_{0}\right\|_{L^{2}}^2 \left\|\nabla B_{0}\right\|^2_{L^{2}}\norm{\Delta B_0}{L^2}^2 e^{C\eta^{-4}\norm{\nabla B_0}{L^2}^4}\right)e^{C\eta^{-2}\norm{\nabla B_0}{L^2}^2}
\end{split}
\een
for all $t>0$. So, there exists a unique global-in-time solution by Proposition \ref{prop1} when
\[
\left(\norm{\Delta\psi_{0}}{L^2}^2+\frac{C}{\eta^{3}}\left\|B_{0}\right\|_{L^{2}}^2 \left\|\nabla B_{0}\right\|^2_{L^{2}}\norm{\Delta B_0}{L^2}^2 e^{C\eta^{-4}\norm{\nabla B_0}{L^2}^4}\right)e^{C\eta^{-2}\norm{\nabla B_0}{L^2}^2} \leq \epsilon \eta^{2},
\]
but we cannot determine the smallness condition only in terms of $\psi_{0}$.
\end{remark}

%%%%%%%%%%%%%%%%%%%%%%%%%
\subsection{Proof of Theorem \ref{Theorem 1.4}}
%%%%%%%%%%%%%%%%%%%%%%%%%
We recall \eqref{H MHD New}:
\begin{subequations}\label{2.5 H MHD New}
\begin{align}
&\partial_{t} u-\nu\Delta u+u\cdot\widetilde{\nabla} u-(\widetilde{\nabla} \times  B)\times B+\widetilde{\nabla} p=0, \label{2.5 H MHD New a} \\
&\partial_{t} B-\eta\Delta B+u\cdot\widetilde{\nabla} B-B\cdot\widetilde{\nabla} u+\widetilde{\nabla} \times ((\widetilde{\nabla} \times  B)\times B)=0, \label{2.5 H MHD New b}\\
&\dv u=\dv  B=0.
\end{align}
\end{subequations}
Using $u=(\partial_2 \phi,-\partial_1\phi,u^3)$ and $B=(\partial_2\psi,-\partial_1\psi,B^3)$, we reformulate \eqref{2.5 H MHD New b} as 
\begin{subequations}\label{2.5-HMHD-alter N}
\begin{align}
&\partial_{t}\psi-\eta\Delta\psi=-\nabla B^3\cdot\nabla^{\perp}\psi+\nabla\phi\cdot\nabla^{\perp}\psi, \label{2.5-HMHD-alter N a}\\
&\partial_{t}B^{3}-\eta\Delta B^{3}=\nabla\Delta\psi\cdot\nabla^{\perp}\psi-\nabla B^{3}\cdot\nabla^{\perp}\phi+\nabla u^{3}\cdot\nabla^{\perp}\psi. 
\end{align} 
\end{subequations}
Theorem \ref{Theorem 1.4} can be proved as Theorem \ref{Theorem 1.3} by using Proposition \ref{prop2} below.

\begin{proposition}\label{prop2}\upshape
Let $u, B\in C([0,T];H^{2}(\mathbb{R}^2))$ be a solution of \eqref{2.5-HMHD-alter N}. Let
\[
\begin{split}
\mathcal{H}_1&=\left(\norm{\nabla B_0}{L^2}^2+\frac{C\mathcal{E}_0\mathcal{E}_2}{\eta\min{(\nu,\eta)}}\right)\exp{\left(\frac{C\mathcal{E}_0}{\nu\eta}\right)},\qquad \mathcal{S}_1= \frac{\mathcal{H}_1}{\eta^2}+\frac{\mathcal{E}_0}{\nu\eta},\\
\mathcal{H}_2&= \left(\norm{B_0}{H^2}^2+\frac{C\mathcal{E}_0\mathcal{E}_2}{\eta\min{(\nu,\eta)}}+\frac{C\mathcal{E}_0\mathcal{E}_2^2}{\eta^4}\right)\exp{\left(\frac{C\mathcal{E}_2}{\eta\min{(\nu,\eta)}}+\frac{C\mathcal{H}_1^2}{\eta^4}\right)}+\norm{u_0}{L^2}^2,
\end{split}
\]
where $\mathcal{E}_0$ and $\mathcal{E}_2$ are defined in \eqref{eq:energy-ineq} and \eqref{eq:2.5D-bound-2}, respectively. There exists $\epsilon>0$ such that if
\eqn \label{smallness proposition 2}
\sup_{t\in[0,T]}\norm{\Delta\psi(t)}{L^2}^2\leq {\epsilon} \eta^2,
\een
$B$ satisfies the following uniform-in-time bounds:
\begin{subequations}
\begin{align}
&\sup_{t\in[0,T]}\norm{\nabla B(t)}{L^2}^2+\eta\int^T_0\norm{\Delta B(t)}{L^2}^2\,dt \leq \mathcal{H}_1, \label{Full H1 bound HMHD} \\
&\sup_{t\in[0,T]}\norm{B(t)}{H^2}^2+\eta\int^{T}_{0}\norm{\nabla B(t)}{H^2}^2\,dt \leq \mathcal{H}_{2}, \label{Full H2 bound HMHD},\\
& \sup_{t\in[0,T]}\norm{\Delta\psi(t)}{L^2}^{2}\leq \norm{\Delta\psi_{0}}{L^2}^{2\exp{(-C\mathcal{S}_1)}}e^{C\mathcal{S}_1(1+\ln \mathcal{H}_{2})}. \label{Full H2 bound psi  HMHD}
\end{align}
\end{subequations}
\end{proposition}

\begin{proof}
By taking the $L^2$ inner product of \eqref{2.5 H MHD New b} with $-\Delta B$, we have
\[
\begin{split}
\frac{1}{2}\frac{d}{dt}\norm{\nabla B}{L^2}^2+\eta\norm{\Delta B}{L^2}^2&= \int(u\cdot\widetilde{\nabla} B-B\cdot\widetilde{\nabla} u )\cdot\Delta B+\int\widetilde{\nabla}\times((\widetilde{\nabla}\times B)\times B)\cdot\Delta B\\
&=\text{(I)+(II)}.
\end{split}
\]
We bound (II) as \eqref{eq:5.8} and we bound
\begin{align*}
\text{(I)}&=-\int(\partial_k u\cdot\widetilde{\nabla} B)\cdot\partial_k B-\int(B\cdot\widetilde{\nabla} u)\cdot\Delta B \leq C\norm{\nabla u}{L^2}\norm{\nabla B}{L^4}^2+C\norm{\nabla u}{L^4}\norm{B}{L^{4}}\norm{\Delta B}{L^2} \\
&\leq C\norm{\nabla u}{L^2}\norm{\nabla B}{L^2}\norm{\Delta B}{L^2}+C\norm{\Delta u}{L^2}\norm{B}{L^2}\norm{\Delta B}{L^2} \\
&\leq \frac{\eta}{4}\norm{\Delta B}{L^2}^2+\frac{C}{\eta}\norm{\nabla u}{L^2}^2\norm{\nabla B}{L^2}^2+\frac{C}{\eta}\norm{B}{L^2}^2\norm{\Delta u}{L^2}^2.
\end{align*}
Hence, if $4C\sqrt{\epsilon}\leq 1$ in \eqref{smallness proposition 2}, we obtain
\[
\frac{d}{dt}\norm{\nabla B}{L^2}^2+\eta\norm{\Delta B}{L^2}^2\leq \frac{C}{\eta}\norm{\nabla u}{L^2}^2\norm{\nabla B}{L^2}^2+\frac{C}{\eta}\norm{B}{L^2}^2\norm{\nabla \omega}{L^2}^2.
\]
By Gr\"onwall's inequality with \eqref{eq:energy-ineq} and \eqref{eq:2.5D-bound-2}, we derive \eqref{Full H1 bound HMHD} as follows:
\[
\begin{split}
& \norm{\nabla B(t)}{L^2}^2+\eta\int^{t}_{0}\norm{\Delta B(\tau)}{L^2}^2\,d\tau \\
& \leq \left[\norm{\nabla B_{0}}{L^2}^2+\frac{C}{\eta}\int^{t}_{0}\norm{B(\tau)}{L^2}^2\norm{\nabla \omega(\tau)}{L^2}^2\,d\tau\right] \exp\left[\frac{C}{\eta}\int^{t}_{0}\norm{\nabla u(\tau)}{L^2}^2\,d\tau\right] \leq \mathcal{H}_{1}.
\end{split}
\]
By taking $L^2$ inner product of \eqref{2.5 H MHD New b} with $\Delta^2 B$, we have
\[
\frac{1}{2}\frac{d}{dt}\norm{\Delta B}{L^2}^2+\eta\norm{\nabla\Delta B}{L^2}^2 =-\int(u\cdot\widetilde{\nabla} B-B\cdot\widetilde{\nabla} u)\cdot\Delta^2 B-\int\widetilde{\nabla}\times((\widetilde{\nabla}\times B)\times B))\cdot\Delta^2 B =\text{(III) +(IV)}.
\]
From \eqref{eq:II3}, we already bound (IV) as 
\[
\text{(IV)}\leq \frac{\eta}{4}\norm{\nabla\Delta B}{L^2}^2+\frac{C}{\eta^3}\norm{\nabla B}{L^2}^2\norm{\Delta B}{L^2}^4,
\]
and we bound
\begin{align*}
\text{(III)} & =-\int(\Delta u\cdot\widetilde{\nabla} B+2\partial_k u\cdot\widetilde{\nabla} \partial_k B)\cdot\Delta B -\int(\partial_k B\cdot\widetilde{\nabla} u+B\cdot\widetilde{\nabla}\partial_k u)\cdot\partial_k\Delta B\\
&\leq C\norm{\Delta u}{L^2}\left(\norm{\nabla B}{L^{\infty}}\norm{\Delta B}{L^2}+\norm{B}{L^\infty}\norm{\nabla\Delta B}{L^2}\right)+C\norm{\nabla u}{L^2}\left(\norm{\Delta B}{L^4}^2+\norm{\nabla B}{L^\infty}\norm{\nabla\Delta B}{L^2}\right) \\
&\leq C\norm{\Delta u}{L^2}\norm{B}{L^2}^{1/2}\norm{\Delta B}{L^2}^{1/2}\norm{\nabla\Delta B}{L^2}+C\norm{\nabla u}{L^2}\norm{\nabla B}{L^2}^{1/2}\norm{\nabla\Delta B}{L^2}^{3/2} \\
&\leq \frac{\eta}{4}\norm{\nabla\Delta B}{L^2}^2+\frac{C}{\eta}\left(\norm{B}{L^2}^2\norm{\Delta u}{L^2}^2+\norm{\Delta u}{L^2}^2\norm{\Delta B}{L^2}^2\right)+\frac{C}{\eta^3}\norm{\nabla u}{L^2}^4\norm{\nabla B}{L^2}^2.
\end{align*}
From these two bounds, we derive
\[
\begin{split}
\frac{d}{dt}\norm{\Delta B}{L^2}^2+\eta\norm{\nabla\Delta B}{L^2}^2 &\leq \frac{C}{\eta}\norm{\Delta u}{L^2}^2\norm{\Delta B}{L^2}^2+\frac{C}{\eta^3}\norm{\nabla B}{L^2}^2\norm{\Delta B}{L^2}^4\\
&+\frac{C}{\eta}\norm{B}{L^2}^2\norm{\Delta u}{L^2}^2+\frac{C}{\eta^3}\norm{\nabla u}{L^2}^4\norm{\nabla B}{L^2}^2.
\end{split}
\]
By Gr\"onwall's inequality and using \eqref{eq:energy-ineq}, \eqref{eq:2.5D-bound-2} and \eqref{Full H1 bound HMHD}, we get 
\eqn \label{Full H2 bound}
\begin{split}
&\norm{\Delta B(t)}{L^2}^2+\eta\int^t_0\norm{\nabla\Delta B(\tau)}{L^2}^2\,d\tau\\
& \leq \left(\norm{\Delta B_0}{L^2}^2 +\frac{C}{\eta}\int^{t}_{0} \norm{B(\tau)}{L^2}^2\norm{\Delta u(\tau)}{L^2}^2d\tau +\frac{C}{\eta^{3}}\int^{t}_{0}\norm{\nabla u(\tau)}{L^2}^4\norm{\nabla B(\tau)}{L^2}^2 d\tau\right)\\
&\times  \exp\left[\frac{C}{\eta}\int^{t}_{0}\norm{\Delta u(\tau)}{L^2}^2d\tau +\frac{C}{\eta^3}\int^{t}_{0}\norm{\nabla B(\tau)}{L^2}^2\norm{\Delta B(\tau)}{L^2}^2 d\tau\right]\\
& \leq \left(\norm{\Delta B_0}{L^2}^2+\frac{C\mathcal{E}_0\mathcal{E}_2}{\eta\min{(\nu,\eta)}}+\frac{C\mathcal{E}_0\mathcal{E}_2^2}{\eta^4}\right)\exp\left[\frac{C\mathcal{E}_2}{\eta\min{(\nu,\eta)}}+\frac{C\mathcal{H}_1^2}{\eta^4}\right].
\end{split}
\een

By \eqref{eq:energy-ineq}, \eqref{Full H1 bound HMHD} and \eqref{Full H2 bound}, we arrive at \eqref{Full H2 bound HMHD}.

\vspace{1ex}

We now bound $\Delta \psi$. By taking $L^2$ inner product of \eqref{2.5-HMHD-alter N a} with $\Delta^2\psi$, we have
\eqn \label{eq:4.18}
\frac{1}{2}\frac{d}{dt}\norm{\Delta\psi}{L^2}^2+\eta\norm{\nabla\Delta\psi}{L^2}^2=-\int (\nabla B^3\cdot\nabla^{\perp}\psi) \Delta^2\psi+\int (\nabla \phi \cdot\nabla^{\perp}\psi) \Delta^{2}\psi.
\een
We already bound the first term on the right-hand side of (\ref{eq:4.18}) in the proof of Proposition \ref{prop1}: 
\[
-\int (\nabla B^3\cdot\nabla^{\perp}\psi) \Delta^2\psi \leq \frac{\eta}{4}\norm{\nabla\Delta\psi}{L^2}^2 +\frac{C}{\eta}\left(\norm{\nabla\psi}{L^\infty}^2+\norm{\Delta\psi}{L^2}^2\right)\norm{\Delta B^3}{L^2}^2.
\]
Similarly, we bound the second term on the right-hand side of (\ref{eq:4.18}) as follows:
\[
\int (\nabla \phi \cdot\nabla^{\perp}\psi) \Delta^{2}\psi \leq \frac{\eta}{4}\norm{\nabla\Delta\psi}{L^2}^2 +\frac{C}{\eta}\left(\norm{\nabla\psi}{L^\infty}^2+\norm{\Delta\psi}{L^2}^2\right)\norm{\Delta \phi}{L^2}^2.
\]
Combining these two bounds, we have
\[
\frac{d}{dt}\norm{\Delta\psi}{L^2}^2+\eta\norm{\nabla\Delta\psi}{L^2}^2\leq \frac{C}{\eta}\left(\norm{\nabla\psi}{L^\infty}^2+\norm{\Delta\psi}{L^2}^2\right)\left(\norm{\Delta B}{L^2}^2+\norm{\nabla u}{L^2}^2\right).
\]
By Lemma \ref{Lemma 2.1} and (\ref{Full H2 bound HMHD}),
\begin{align*}
\frac{d}{dt}\norm{\Delta\psi}{L^2}^2+\eta\norm{\nabla\Delta\psi}{L^2}^2 &\leq \frac{C}{\eta}\left[1+\ln{\left(\frac{\norm{\nabla\psi}{H^2}^2}{\norm{\Delta \psi}{L^2}^2}\right)}\right]\left(\norm{\Delta B}{L^2}^2+\norm{\nabla u}{L^2}^2\right)\norm{\Delta\psi}{L^2}^2 \\
&\leq \frac{C}{\eta}\left(1+\ln \mathcal{H}_{2}+\ln{\norm{\Delta\psi}{L^2}^{-2}}\right)\left(\norm{\Delta B}{L^2}^2+\norm{\nabla u}{L^2}^2\right)\norm{\Delta\psi}{L^2}^2.
\end{align*}
By setting 
\eqn \label{F HMHD}
F(t)=\norm{\Delta\psi(t)}{L^2}^{2}\exp\left[-\frac{C}{\eta}(1+\ln \mathcal{H}_{2})\int^{t}_{0}\left(\norm{\Delta B(\tau)}{L^2}^2+\norm{\nabla u(\tau)}{L^2}^2\right)\,d\tau\right]
\een
which can be written as 
\[
\frac{d}{dt}F \leq \frac{C}{\eta}\left(\norm{\Delta B}{L^2}^2+\norm{\nabla u}{L^2}^2\right)F\ln{F^{-1}}.
\]
Hence we derive
\[
\frac{d}{dt}\left(\ln \ln F^{-1}\right)\geq -\frac{C}{\eta}\left(\norm{\Delta B}{L^2}^2+\norm{\nabla u}{L^2}^2\right).
\]
Since 
\[
\frac{1}{\eta}\int^t_0\left(\norm{\Delta B(\tau)}{L^2}^2+\norm{\nabla u(\tau)}{L^2}^2\right)\,d\tau\leq \frac{\mathcal{H}_1}{\eta^2}+\frac{\mathcal{E}_0}{\nu\eta}= \mathcal{S}_1
\]
from \eqref{Full H1 bound HMHD} and \eqref{eq:energy-ineq}, $F$ is bounded as 
\[
F(t)\leq \exp{\left[\ln{F_0}e^{-C\mathcal{S}_1}\right]}=F_0^{\exp{(-C\mathcal{S}_1)}}
\]
and so (\ref{F HMHD}) gives 
\[
\sup_{t\in[0,T]}\norm{\Delta\psi(t)}{L^2}^{2}\leq \norm{\Delta\psi_{0}}{L^2}^{2\exp{(-C\mathcal{S}_1)}}e^{C\mathcal{S}_1(1+\ln \mathcal{H}_{2})}.
\]
This completes the proof of Proposition \ref{prop2}.
\end{proof}

%%%%%%%%%%%%%%%%%%%%
\section{Decay rate of $B+\omega$}\label{sec:5}
%%%%%%%%%%%%%%%%%%%
We now find a weak solution with $B+\omega$ decaying in time when $\nu=\eta>0$. Since $(u_0,B_0,\omega_0)\in L^{2}$ with \eqref{eq:cor-cond2}, the following uniform-in-time bound holds for all $t>0$ by Corollary \ref{Corollary 1.2}:
\eqn \label{B plus omega smallness}
\begin{split}
\norm{(B+\omega)(t)}{L^2}^2+\nu\int^{t}_{0}\norm{\nabla(B+\omega)(\tau)}{L^2}^2\,d\tau \leq \norm{B_0+\omega_0}{L^2}^{2} \exp\left(\frac{C}{\nu^{4}}(\mathcal{E}_0+\mathcal{E}_0^2)\right)=C_{0}.
\end{split}
\een
Therefore, we can find a global weak solution of (\ref{H MHD}) with the energy inequality \eqref{eq:energy-ineq} and \eqref{B plus omega smallness}. Moreover, we can derive decay rates of $\norm{(B+\omega)(t)}{L^2}$ using the refined Fourier splitting method developed in \cite{Bae Jung Shin}.

\begin{theorem} \upshape \label{Theorem 5.1}
Let $\nu=\eta>0$ and let $(u_0,B_0,\omega_0)\in L^{2}(\mathbb{R}^{3})$ satisfy $\dv u_0=\dv  B_0=0$ and \eqref{eq:cor-cond}. There exists a global weak solution $(u,B,\omega)\in L^\infty([0,\infty);L^2(\mathbb{R}^3))\cap L^2([0,\infty);\dot{H}^1(\mathbb{R}^3))$ satisfying \eqref{eq:energy-ineq} and \eqref{B plus omega smallness} for almost every $t\in[0,\infty)$. Furthermore, if
\eqn \label{low frequency B omega}
\sup_{|\xi|\leq 1}|\xi|^{\sigma}|\widehat{B_0+\omega_0}(\xi)| \leq C
\een
for $\sigma\in [-1,1]$, then $B+\omega$ has the following decay rates:
\begin{equation}\label{eq:decay}
\norm{(B+\omega)(t)}{L^2}^2\leq C(1+t)^{-\frac{3}{2}+\sigma}.
\end{equation}
\end{theorem}

To prove Theorem \ref{Theorem 5.1}, we use the following Fourier splitting type lemma.

\begin{lemma}\cite[Lemma 2.1]{Bae Jung Shin} \upshape \label{Lemma 5.1}
Let $f$ be a smooth function satisfying 
\eqn \label{condition to decay}
\frac{d}{dt}\|f(t)\|_{L^{2}}^2 + \nu\|\nabla f(t)\|_{L^{2}}^2 \le 0
\een
for all $t>0$. Suppose there exists a positive constant $C_{\ast}>0$ and $\sigma < \frac 32$ such that
\[
\sup_{0\leq t<\infty}\sup_{|\xi|\leq 1}|\xi|^\sigma \left|\widehat f(t,\xi)\right|  \le C_{\ast}.
\]
Then, for all $t\geq N-1$, where $N-1$ is a non-negative constant with $N>\frac{3}{2}-\sigma$, 
\[
\|f(t)\|_{L^{2}}^2  \le C(1+t)^{-\frac{3}{2}+\sigma}
\]
\end{lemma}

In the proof of Theorem \ref{Theorem 5.1}, we also use the following two inequalities. For all $x>0$ and $p>0$
\begin{equation}\label{heat_inf_est}
|x^p e^{-ax^2}| \le  |x^p e^{-ax^2}| \big|_{x = \sqrt{\frac{p}{2a}}} = \left(\frac{p}{2a} \right)^{\frac p2} e^{-\frac p2}.
\end{equation} 
For $0<\alpha,\beta<1$ with $\alpha+\beta=1$
\eqn \label{Beta inequlaity}
\int^{t}_{0}(t-\tau)^{-\alpha}\tau^{-\beta}d\tau=\int_0^1 (1-\theta)^{-\alpha}\theta^{-\beta}\,d\theta = \mathcal{B}\left(\alpha,\beta\right),
\een
where $\mathcal{B}(\cdot, \cdot)$ is the beta function.

%%%%%%%%%%%%%%%%%%%%%%%%%
\subsection{Proof of Theorem \ref{Theorem 5.1}}
%%%%%%%%%%%%%%%%%%%%%%%%%
Although we have a uniform bound (\ref{B plus omega smallness}), $B+\omega$ does not satisfies (\ref{condition to decay}). To overcome this difficulty, we use 
\[
\Gamma(t,x)=(B+\omega)(t,x)\exp\left[-\lambda\int^{t}_{0}\norm{u(\tau)}{\text{BMO}}^2\,d\tau\right],
\]
where $\lambda>0$ will be chosen later. Then, from \eqref{eq:B+omega2}, we derive the equation of $\Gamma$: 
\[
\partial_t \Gamma+\lambda\norm{u(t)}{\text{BMO}}^2\Gamma-\nu\Delta\Gamma+(u\cdot\nabla)\Gamma=(\Gamma\cdot\nabla)u.
\]
By using (\ref{Hardy 1}) and (\ref{BMO 1}), 
\begin{align*}
\frac{1}{2}\frac{d}{dt}\norm{\Gamma}{L^2}^2 &+\lambda\norm{u}{\text{BMO}}^2\norm{\Gamma}{L^2}^2+\nu\norm{\nabla\Gamma}{L^2}^2  =\int \left(\Gamma\cdot\nabla u\right)\cdot \Gamma =-\int(\Gamma\cdot\nabla\Gamma)\cdot u \\
&\leq \norm{u}{\text{BMO}}\norm{\Gamma\cdot\nabla\Gamma}{\mathcal{H}}\leq C\norm{u}{\text{BMO}}\norm{\Gamma}{L^2}\norm{\nabla\Gamma}{L^2}\leq \frac{\nu}{2}\norm{\nabla\Gamma}{L^2}^2+\frac{C}{\nu}\norm{u}{\text{BMO}}^2\norm{\Gamma}{L^2}^2 .
\end{align*}
By setting $\lambda=C\nu^{-1}$, we obtain 
\[
\frac{d}{dt}\norm{\Gamma}{L^2}^2+\nu\norm{\nabla\Gamma}{L^2}^2\leq0.
\]
To show the decay rates of $B+\omega$ using $\Gamma$, we bound the exponential term in $\Gamma$. Since 
\[
\begin{split}
\int^{t}_{0}\norm{u(\tau)}{\text{BMO}}^2\,d\tau &\leq C \int^{t}_{0}\norm{\nabla u(\tau)}{L^2}\norm{\nabla \omega(\tau)}{L^2}\,d\tau \\
&\leq \frac{C}{\nu}\left(C_{0}+\left\|u_{0}\right\|^{2}_{L^{2}} +\left\|B_{0}\right\|^{2}_{L^{2}}\right)^{\frac{1}{2}} \left(\left\|u_{0}\right\|^{2}_{L^{2}} +\left\|B_{0}\right\|^{2}_{L^{2}}\right)^{\frac{1}{2}},
\end{split}
\]
where $C_{0}$ is defined in (\ref{B plus omega smallness}), we find  
\[
\exp\left[\frac{C}{\nu}\int^{t}_{0}\norm{u(\tau)}{\text{BMO}}^2\,d\tau\right]\leq \exp \left[\frac{C}{\nu^{2}}\left(C_{1}+\left\|u_{0}\right\|^{2}_{L^{2}} +\left\|B_{0}\right\|^{2}_{L^{2}}\right)^{\frac{1}{2}} \left(\left\|u_{0}\right\|^{2}_{L^{2}} +\left\|B_{0}\right\|^{2}_{L^{2}}\right)^{\frac{1}{2}}\right]=C_{\ast}
\]
from which we have the point-wise bound of $\Gamma$ and $B+\omega$:
\[
|\Gamma(t,x)|\leq |(B+\omega)(t,x)| \leq C_{\ast}|\Gamma(t,x)|
\]
and thus we have 
\eqn \label{Gamma B omega L2}
\norm{\Gamma(t)}{L^2}\leq \norm{(B+\omega)(t)}{L^2} \leq C_{\ast}\norm{\Gamma(t)}{L^2} .
\een
Moreover, since the exponential term in $\Gamma$ does not depend on the spatial variables, 
\eqn \label{Gamma B omega}
\sup_{|\xi|\leq 1}|\xi|^{\sigma}|\widehat{\Gamma}(t,\xi)|\leq \sup_{|\xi|\leq 1}|\xi|^{\sigma}|\widehat{B+\omega}(t,\xi)|\leq C_{\ast}\sup_{|\xi|\leq 1}|\xi|^{\sigma}|\widehat{\Gamma}(t,\xi)|.
\een

To derive decay rates of $B+\omega$, we express the equation of $B+\omega$ using the Fourier transform:
\[
\widehat{B+\omega}(t,\xi)= e^{-\nu t|\xi|^2}\widehat{B_0+\omega_0}(\xi)+\int^{t}_{0}e^{-\nu(t-\tau)|\xi|^{2}}i\xi\cdot\mathcal{F}\left(-u\otimes(B+\omega)+(B+\omega)\otimes u\right)(\tau,\xi)\,d\tau.
\]
From now on, we suppress the $\nu$-dependance on the bounds of solutions for simplicity. 

\vspace{1ex}

\noindent
$\blacktriangleright$ Suppose (\ref{low frequency B omega}) holds with $\sigma=1$. Then for $|\xi|\leq 1$,
\begin{align*}
|\xi||\widehat{B+\omega}(t,\xi)| &\leq |\xi||\widehat{B_0+\omega_0}(\xi)|+\int^{t}_{0}e^{-\nu(t-\tau)|\xi|^{2}}|\xi|^2|\mathcal{F}\left(-u\otimes(B+\omega)+(B+\omega)\otimes u\right)(\tau,\xi)|\,d\tau \\
&\leq |\xi||\widehat{B_0+\omega_0}(\xi)|+C\sup_{\tau\in[0,t]}\norm{u(\tau)}{L^2}\norm{(B+\omega)(\tau)}{L^2}\leq C.
\end{align*}
Hence, by Lemma \ref{Lemma 5.1} with \eqref{Gamma B omega L2} and (\ref{Gamma B omega}),  we derive 
\eqn \label{eq:decay1}
\norm{\Gamma(t)}{L^2}^2\leq C(1+t)^{-\frac{1}{2}}, \quad \norm{(B+\omega)(t)}{L^2}^2\leq C(1+t)^{-\frac{1}{2}}.
\een

\noindent
$\blacktriangleright$ We now handle the case $\sigma\in (-1,1)$. We note that  when $\sigma_1\geq\sigma_2$,
\begin{equation} \label{low frequency embedding}
\sup_{|\xi|\leq 1}|\xi|^{\sigma_1}|\widehat{B_0+\omega_0}(\xi)| \leq \sup_{|\xi|\leq 1}|\xi|^{\sigma_2}|\widehat{B_0+\omega_0}(\xi)|.
\end{equation}
Hence, if (\ref{low frequency B omega}) holds for $\sigma\in(-1,1)$, $\norm{(B+\omega)(t)}{L^2}$ decays as \eqref{eq:decay1}. By (\ref{heat_inf_est}), for $|\xi|\leq1$,
\begin{align*}
|\xi|^{\sigma}|\widehat{B+\omega}(t,\xi)| &\leq |\xi|^{\sigma}|\widehat{B_0+\omega_0}(\xi)|+\int^{t}_{0}e^{-\nu(t-\tau)|\xi|^{2}}|\xi|^{1+\sigma} \norm{u(\tau)}{L^2}\norm{(B+\omega)(\tau)}{L^2}\,d\tau \\
&\leq |\xi|^{\sigma}|\widehat{B_0+\omega_0}(\xi)|+C\int^{t}_{0} (t-\tau)^{-\frac{1+\sigma}{2}}(1+\tau)^{-\frac{1}{4}}\,d\tau.
\end{align*}
When $\sigma\in [\frac{1}{2},1)$, the time integral part is estimated by using (\ref{Beta inequlaity})
\[
\int_0^t (t-\tau)^{-\frac{1+\sigma}{2}}(1+\tau)^{-\frac 14}\,d\tau\le \int_0^t (t-\tau)^{-\frac{1+\sigma}{2}}(1+\tau)^{-\frac {1-\sigma}{2}}\,d\tau\le \mathcal{B}\left(\frac{1-\sigma}{2}, \frac{1+\sigma}{2} \right).
\]
So, Lemma \ref{Lemma 5.1} with \eqref{Gamma B omega L2} and \eqref{Gamma B omega} implies \eqref{eq:decay} for $\sigma \in [\frac{1}{2},1)$. Next, suppose \eqref{low frequency B omega} holds for $\sigma\in[0,\frac{1}{2})$. Then, from \eqref{low frequency embedding} and \eqref{eq:decay} with $\sigma=\frac{1}{2}$, 
\[
\norm{(B+\omega)(t)}{L^2}^2\leq C(1+t)^{-1}.
\]
Following the same computation above, we obtain
\[
|\xi|^{\sigma}|\widehat{B+\omega}(t,\xi)|\leq |\xi|^{\sigma}|\widehat{B_0+\omega_0}(\xi)|+C\int^{t}_{0}(t-\tau)^{-\frac{1+\sigma}{2}}(1+\tau)^{-\frac{1}{2}}\,d\tau\leq C
\]
when $\sigma\in [0,\frac{1}{2})$ from which \eqref{eq:decay} holds. Repeating this process to $\sigma\in [-\frac{1}{2},0)$ and $\sigma\in(-1,-\frac{1}{2})$, we conclude that \eqref{eq:decay} holds when $\sigma\in(-1,1)$.

\vspace{1ex}

\noindent
$\blacktriangleright$ We finally deal with the case  $\sigma=-1$. Then, from \eqref{low frequency embedding} and \eqref{eq:decay} for $\sigma\in(-1,-\frac{1}{2})$, we have
\[
\norm{(B+\omega)(t)}{L^2}^2\leq C(1+t)^{-2(1+\kappa)}
\]
for any $\kappa\in (0,\frac{1}{4})$. Using this, for $|\xi|\leq1$,
\begin{align*}
|\xi|^{-1}|\widehat{B+\omega}(t,\xi)| &\leq |\xi|^{-1}|\widehat{B_0+\omega_0}(\xi)|+\int^{t}_{0}e^{-\nu(t-\tau)|\xi|^{2}}\norm{u(\tau)}{L^2}\norm{(B+\omega)(\tau)}{L^2}\,d\tau \\
&\leq |\xi|^{-1}|\widehat{B_0+\omega_0}(\xi)|+C\int^{t}_{0} (1+\tau)^{-1-\kappa}\,d\tau \leq C.
\end{align*}
By Lemma \ref{Lemma 5.1} with \eqref{Gamma B omega L2} and (\ref{Gamma B omega}), we complete the proof of Theorem \ref{Theorem 5.1}.

\begin{remark}\upshape
Theorem \ref{Theorem 5.1} implies that if $B_0+\omega_0$ satisfies  \eqref{low frequency B omega}, $\norm{(B+\omega)(t)}{L^2}$ decays as \eqref{eq:decay} even though $(u,B,\omega)$ does not decay in time. In contrast, \cite[Theorem 1.3]{Bae Jung Shin} indicates that if 
\[
\sup_{|\xi|\leq 1}|\xi|^{\sigma_1}|\widehat{u}_0(\xi)|\leq C,\quad \sup_{|\xi|\leq 1}|\xi|^{\sigma_2}|\widehat{B}_0(\xi)|\leq C
\]
for $\sigma_1\in[-1,1]$ and $\sigma_2\in[-1,0]$, then 
$(u,B,\omega)$ decay as follows:
\[
\norm{u(t)}{L^2}^2\leq C(1+t)^{-\frac{3}{2}+\sigma_1},\quad \norm{B(t)}{L^2}^2\leq C(1+t)^{-\frac{3}{2}+\sigma_2},\quad \norm{\omega(t)}{L^2}^2\leq C(1+t)^{-\frac{5}{2}+\sigma_1}.
\]
\end{remark}

%%%%%%%%%%%%%%%
\section*{Acknowledgments}
%%%%%%%%%%%%%%%

H. Bae was supported by the National Research Foundation of Korea (NRF) grant funded by the Korea government (MSIT) (2022R1A4A1032094 and RS-2024-00341870). 

K. Kang was supported by the National Research Foundation of Korea(NRF) grant funded by the Korea government (MSIT) (RS-2024-00336346 and RS-2024-00406821). 

\section*{Data Availability}

Data sharing is not applicable to this article as no datasets were created or analyzed in this study.

\section*{Declarations}

\textbf{Conflict of interest} The authors declare that there is no conflict of interest.


\begin{thebibliography}{99}

\bibitem{Acheritogaray}
     \newblock M. Acheritogaray, P. Degond, A. Frouvelle, J-G. Liu.
     \newblock Kinetic formulation and global existence for the Hall- Magneto-hydrodynamics system.
     \newblock \emph{Kinet. Relat. Models} {\bf 4} (2011) 901--918.





\bibitem{Bae Jung Shin}
     \newblock H. Bae, J. Jung, J. Shin.
     \newblock On the temporal estimates for the incompressible Navier-Stokes equations and Hall-magnetohydrodynamic equations.
       \newblock 	arXiv:2404.16290.



\bibitem{Bae Kang}
     \newblock H. Bae, K. Kang.
     \newblock On the Existence and Temporal Asymptotics of Solutions for the Two and Half Dimensional Hall MHD. 
       \newblock \emph{J. Math. Fluid Mech.} {\bf25} (2023). no. 2, Paper No. 24, 30 pp.


\bibitem{Balbus}
     \newblock S.A. Balbus, C. Terquem.
     \newblock Linear analysis of the Hall effect in protostellar disks.
     \newblock \emph{Astrophys. J.} {\bf 552} (2001) 235--247.
     


\bibitem{Brezis Nirenberg}
     \newblock H. Brezis, L. Nirenberg.
     \newblock Degree theory and BMO; part I: Compact manifolds without boundaries.
     \newblock \emph{Selecta Mathematica.} {\bf1} (1995), no. 2, 197--263.




\bibitem{Chae Degond Liu}
     \newblock D. Chae, P. Degond, J-G. Liu.
     \newblock Well-posedness for Hall-magnetohydrodynamics.
     \newblock \emph{Ann. Inst. H. Poincar\'e Anal. Non Lin\'eaire} {\bf 31} (2014), no. 3, 555--565.

    
    
\bibitem{Chae Lee}
     \newblock D. Chae, J. Lee.
     \newblock On the blow-up criterion and small data global existence for the Hall-magnetohydrodynamics.
     \newblock  \emph{J. Differential Equations} {\bf 256} (2014), no. 11, 3835--3858.
     
     
     
\bibitem{Chae Schonbek}
     \newblock D. Chae, M. Schonbek. 
     \newblock On the temporal decay for the Hall-magnetohydrodynamic equations.
     \newblock \emph{J. Differential Equations} {\bf 255} (2013) 3971--3982.
     


\bibitem{Chae Wan Wu}
     \newblock D. Chae, R. Wan, J. Wu.
     \newblock Local well-posedness for the Hall-MHD equations with fractional magnetic diffusion.
     \newblock \emph{J. Math. Fluid Mech.} {\bf 17} (2015), no. 4, 627--638.
     
     
\bibitem{Chae Weng}
     \newblock D. Chae, S. Weng.
     \newblock Singularity formation for the incompressible Hall-MHD equations without resistivity.
     \newblock  \emph{Ann. Inst. H. Poincar\'e Anal. Non Lin\'eaire} {\bf 33} (2016), no. 4, 1009--1022.



\bibitem{Chae Wolf}
	\newblock D. Chae, J. Wolf.
	\newblock On partial regularity for the steady Hall magnetohydrodynamics system.
	\newblock \emph{Comm. Math. Phys.} {\bf 339} (2015), no. 3, 1147--1166.



\bibitem{Chae Wolf 1}
     \newblock D. Chae, J. Wolf.
     \newblock On partial regularity for the 3D nonstationary Hall magnetohydrodynamics equations on the plane. 
     \newblock \emph{SIAM J. Math. Anal.} {\bf48} (2016), no. 1, 443--469. 



\bibitem{Chae Wolf 2}
     \newblock D. Chae, J. Wolf.
     \newblock Regularity of the 3D stationary hall magnetohydrodynamic equations on the plane.
     \newblock \emph{Comm. Math. Phys.} {\bf354} (2017), no. 1, 213--230.



\bibitem{Coifman}
	\newblock	R. Coifman, P.-L. Lions, Y. Meyer, S. Semmes.
	\newblock	Compensated compactness and Hardy spaces.
	\newblock	\textit{J. Math. Pures Appl.} (9) \textbf{72} (1993), no.3, 247–286.



\bibitem{Dai 1}
     \newblock M. Dai.
     \newblock Local well-posedness of the Hall-MHD system in $H^{s}(\mathbb{R}^{n})$ with $s>\frac{n}{2}$.
     \newblock \emph{Math. Nachr.} {\bf 293} (2020), no. 1, 67--78.  


\bibitem{Dai 2}
     \newblock M. Dai.
     \newblock Local well-posedness for the Hall-MHD system in optimal Sobolev spaces.
     \newblock  \emph{J. Differential Equations} {\bf289} (2021), 159--181.  



\bibitem{Dai 3}
     \newblock M. Dai, H. Liu.
     \newblock Long time behavior of solutions to the 3D Hall-magneto-hydrodynamics system with one diffusion.
     \newblock \emph{J. Differential Equations} {\bf 266} (2019), no. 11, 7658--7677. 





\bibitem{Danchin Tan 2}
     \newblock R. Danchin, J. Tan. 
     \newblock On the well-posedness of the Hall-magnetohydrodynamics system in critical spaces.
     \newblock \emph{Comm. Partial Differential Equations} {\bf46} (2021), no. 1, 31--65. 




\bibitem{Danchin Tan 3}
     \newblock R. Danchin, J. Tan. 
     \newblock The global solvability of the Hall-magnetohydrodynamics system in critical Sobolev spaces.
     \newblock \emph{Commun. Contemp. Math.} {\bf24} (2022), no. 10, Paper No. 2150099, 33 pp.




\bibitem{Duoandikoetxea}
     \newblock J. Duoandikoetxea. 
     \newblock Fourier Analysis.
     \newblock Graduate Studies in Mathematics, 29. \emph{American Mathematical Society, Providence, RI}, 2001. xviii+222 pp.



\bibitem{Evans}
\newblock L. Evans.
\newblock Partial differential equations. Graduate Studies in Mathematics, 19.
\newblock \emph{American Mathematical Society, Providence, RI,} 1998. xviii+662 pp. 




\bibitem{Fefferman}
     \newblock C. Fefferman, E. M. Stein. 
     \newblock $H^{p}$ spaces of several variables.
     \newblock \emph{Acta Math.} {\bf129} (1972), no. 3-4, 137--193.


     
     
\bibitem{Forbes}
     \newblock T.G. Forbes.
     \newblock Magnetic reconnection in solar flares.
     \newblock \emph{Geophys. Astrophys. Fluid Dyn.} {\bf 62} (1991) 15--36. 
     


\bibitem{He Huang Wang}
	\newblock	C. He, X. Huang, Y. Wang.
	\newblock	On some new global existence results for 3D magnetohydrodynamic equations.
	\newblock	\textit{Nonlinearity} \textbf{27} (2014), no.2, 343–352.



\bibitem{Homann}
     \newblock H. Homann, R. Grauer.
     \newblock Bifurcation analysis of magnetic reconnection in Hall-MHD systems.
     \newblock \emph{Phys. D} {\bf 208} (2005) 59--72.




\bibitem{Jeong Oh}
     \newblock I. Jeong, S. Oh.
     \newblock On the Cauchy problem for the Hall and electron magnetohydrodynamic equations without resistivity I: Illposedness near degenerate stationary solutions.
     \newblock {\emph Ann. PDE} {\bf 8} (2022), no. 2, Paper No. 15, 106 pp.

     
     
     

\bibitem{John Nirenberg}
     \newblock F. John, L, Nirenberg.
     \newblock On functions of bounded mean oscillation.
     \newblock \emph{Comm. Pure Appl. Math.} {\bf14} (1961), no. 3, 415--426. 




\bibitem{Kozono Taniuchi}
     \newblock H. Kozono, Y. Taniuchi. 
     \newblock Bilinear estimates in BMO and the Navier-Stokes equations.
     \newblock \emph{Math. Z.} {\bf 235} (2000),173--194.


\bibitem{Kwak}
     \newblock M. Kwak, B.  Lkhagvasuren. 
     \newblock Global wellposedness for Hall-MHD equations.
     \newblock  \emph{Nonlinear Anal.} {\bf 174} (2018), 104--117.



\bibitem{Lighthill}
     \newblock M.J. Lighthill.
     \newblock Studies on magneto-hydrodynamic waves and other anisotropic wave motions.
     \newblock  \emph{Philos. Trans. R. Soc. Lond. Ser. A} {\bf 252} (1960) 397--430.




\bibitem{Meyrand}
     \newblock R. Meyrand, S. Galtier.
     \newblock Spontaneous Chiral Symmetry Breaking of Hall Magnetohydrodynamic Turbulence.
     \newblock \emph{Phys. Rev. Lett.} {\bf109} (2012), 194501.


\bibitem{Mininni}
     \newblock P.D. Mininni, D.O. G\'omez, S.M. Mahajan.
     \newblock Dynamo action in magnetohydrodynamics and Hall magnetohydrodynamics.
     \newblock \emph{Astrophys. J.} {\bf 587} (2003) 472--481.




\bibitem{Poly}
	\newblock J. M. Polygiannakis, X. Moussas.
	\newblock A review of magneto-vorticity induction in Hall-MHD plasmas.
	\newblock \emph{Plasma Phys. Control. Fusion} {\bf 43} (2001) 195--221.


\bibitem{Rahman}
	\newblock	M. M. Rahman, K. Yamazaki.
	\newblock	Remarks on the global regularity issue of the two-and-a-half-dimensional Hall-magnetohydrodynamics system.
	\newblock	\textit{Z. Angew. Math. Phys.} \textbf{73} (2022), no.5, Paper No. 217, 29 pp.
	

%\bibitem{Schonbek}
%     \newblock E. M. Schonbek.
%     \newblock $L^{2}$ decay for weak solutions of the Navier-Stokes equations.
%     \newblock \emph{Arch. Rational Mech. Anal.} {\bf88} (1985), no. 3, 209--222.

	
\bibitem{Shalybkov}
     \newblock D.A. Shalybkov, V.A. Urpin.
     \newblock The Hall effect and the decay of magnetic fields.
     \newblock \emph{Astron. Astrophys.} {\bf 321} (1997) 685--690.



\bibitem{Shay}
     \newblock M. A. Shay, J. F. Drake, R. E. Denton, D. Biskamp.
     \newblock Structure of the dissipation region during collisionless magnetic reconnection.
     \newblock  \emph{Journal of Geophysical Research} {\bf 103} (1998), no A5, 9165--9176. 


\bibitem{Tan}
     \newblock J. Tan.
     \newblock New energy functionals for the incompressible Hall-MHD system. 
     \newblock  2022. ffhal-03855638.
     

\bibitem{Turner}
	\newblock L. Turner.
	\newblock Hall effects on magnetic relaxation.
	\newblock \emph{IEEE Trans. Plasma Sci.} {\bf 14} (1986) 849--857.     
     

\bibitem{Wardle}
     \newblock M. Wardle.
     \newblock Star formation and the Hall effect.
     \newblock \emph{Astrophys. Space Sci.} {\bf 292} (2004) 317--323.    
     

\bibitem{Yamazaki}
     \newblock K. Yamazaki.  
     \newblock Irreducibility of the three, and two and a half dimensional Hall-magnetohydrodynamics system.
     \newblock \emph{Phys. D} {\bf401} (2020), 132199, 21 pp.
     


\bibitem{Ye}
	\newblock Z. Ye.
	\newblock Regularity criterion for the 3D Hall-magnetohydrodynamic equations involving the vorticity.
	\newblock \emph{Nonlinear Anal.} {\bf 144} (2016) 182--193.

\end{thebibliography}
\end{document}